\documentclass[a4paper,10pt,american]{amsart}

\usepackage[utf8]{inputenc}
\usepackage{amssymb}
\usepackage{amsmath,bm}
\usepackage{times}
\usepackage{dsfont}
\usepackage{stmaryrd}
\usepackage{esint}
\usepackage{mathtools} 

\usepackage[left=3cm,right=3cm]{geometry}
\usepackage{enumitem}
\usepackage{xcolor}
\usepackage[colorlinks,pdfpagelabels,pdfstartview = FitH,bookmarksopen
= true,bookmarksnumbered = true,linkcolor = blue,plainpages =
false,hypertexnames = false,citecolor = red,pagebackref=false]{hyperref}

\allowdisplaybreaks

\newtheorem{theorem}{Theorem}
\newtheorem{lemma}[theorem]{Lemma}

\theoremstyle{definition}
\newtheorem{definition}[theorem]{Definition}

\theoremstyle{remark}
\newtheorem{remark}[theorem]{Remark}
\newtheorem{example}[theorem]{Example}

\numberwithin{theorem}{section}
\numberwithin{equation}{section}

\renewenvironment{proof}[1][\proofname]{{\rm\bfseries #1. }}{\qed}

\newcommand{\N}{\mathds{N}}
\newcommand{\Z}{\mathds{Z}}
\newcommand{\R}{\mathds{R}}

\newcommand{\W}{\mathcal{W}}

\renewcommand{\d}{\mathrm{d}}
                  
\newcommand{\dx}{\mathrm{d}x}
\newcommand{\dy}{\mathrm{d}y}
\newcommand{\dz}{\mathrm{d}z}
\newcommand{\dt}{\mathrm{d}t}

\newcommand{\loc}{\textnormal{loc}}

\renewcommand{\epsilon}{\varepsilon}
\renewcommand{\rho}{\varrho}

\DeclareMathOperator{\spt}{spt}
\DeclareMathOperator{\diam}{diam}
\DeclareMathOperator{\dist}{dist}

\DeclareMathOperator{\cp}{cap}
\DeclareMathOperator{\Lip}{Lip}

\def\Xint#1{\mathchoice
    {\XXint\displaystyle\textstyle{#1}}%
    {\XXint\textstyle\scriptstyle{#1}}%
    {\XXint\scriptstyle\scriptscriptstyle{#1}}%
    {\XXint\scriptscriptstyle\scriptscriptstyle{#1}}%
    \!\int}
\def\XXint#1#2#3{\setbox0=\hbox{$#1{#2#3}{\int}$}
    \vcenter{\hbox{$#2#3$}}\kern-0.5\wd0}
\def\bint{\Xint-}
\def\dashint{\Xint{\raise4pt\hbox to7pt{\hrulefill}}}

\def\XXiint#1#2#3{\setbox0=\hbox{$#1{#2#3}{\iint}$}
    \vcenter{\hbox{$#2#3$}}\kern-0.5\wd0}

\newenvironment{todorev}{\color{cyan}}{\color{black}}
\newcommand{\btodo}{\begin{todorev}}
\newcommand{\etodo}{\end{todorev}}

\newenvironment{samuelerev}{\color{purple}}{\color{black}}
\newcommand{\bsamr}{\begin{samuelerev}}
\newcommand{\esamr}{\end{samuelerev}}

\begin{document}
\title[Global higher integrability and Hardy inequalities for double-phase functionals]{Global higher integrability and Hardy inequalities \\ for double-phase functionals under a capacity density condition}

\author[F. Bäuerlein]{Fabian Bäuerlein}
\address[F. Bäuerlein]{Universität Salzburg, Fachbereich Mathematik,
Hellbrunner Str. 34, 5020 Salzburg, Austria.}
\email{fabian.baeuerlein@plus.ac.at}

\author[L. Kurki Schätzler]{Leah Kurki Schätzler}
\address[L. Kurki Schätzler]{University of Lisbon, Department of Mathematics, Campo Grande, 1749 016 Lisboa, Portugal}
\email{laschatzler@ciencias.ulisboa.pt}

\author[S. Ricc\`o]{Samuele Ricc\`o}
\address[S. Ricc\`o]{TU Wien, Institute of Analysis and Scientific Computing, Wiedner Hauptstraße 8-10/104, 1040 Vienna, Austria.}
\email{samuele.ricco@tuwien.ac.at}


\begin{abstract}
    We prove global higher integrability for functionals of double-phase type under a uniform local capacity density condition on the complement of the considered domain $\Omega \subset \R^n$.
    In this context, we investigate a new natural notion of variational capacity associated to the double-phase integrand.
    Under the related fatness condition for the complement of $\Omega$, we establish an integral Hardy inequality.
    Further, we show that fatness of $\R^n \setminus \Omega$ is equivalent to a boundary Poincaré inequality, a pointwise Hardy inequality and to the local uniform $p$-fatness of $\R^n \setminus \Omega$.
    We provide a counterexample that shows that the expected Maz'ya type inequality--a key intermediate step toward global higher integrability--does not hold with the notion of capacity involving the double-phase functional itself.
    \\[0.5cm]
\end{abstract}

\maketitle


\section{Introduction}

In the present paper, we consider the global higher integrability of quasi-minimizers to functionals of double-phase type in a bounded and open set $\Omega\subset \R^n$.
To this end, not only the regularity of the boundary datum, but also the regularity of the underlying domain itself is crucial.
In the special case of minimizers of the $p$-energy
$$
    \int_\Omega |\nabla u|^p \,\dx,
$$
it is well-known that the uniform $p$-capacity density condition, also known as uniform $p$-fatness, of the complement $\R^n \setminus \Omega$ is optimal, see \cite{KK94}.
Thus, the question arises what the optimal condition on $\Omega$ is if the integrand in the preceding functional is replaced by the double-phase integrand
\begin{equation}
    \varphi(x,t) = t^p + a(x) t^q,
    \label{eq:integrand}
\end{equation}
where $1 < p< q < +\infty$, and the modulating coefficient $a\colon \R^n \to \R_{\geq 0}$ satisfies
\begin{equation}\label{as:a}
    a\in C^{0,\alpha}(\R^n) \qquad\text{for some}\quad \alpha \in (0,1].
\end{equation}
Starting from \cite{M89}, \textsc{P. Marcellini} developed the theory for functionals with non-standard growth, also known as $(p,q)$-growth, while in \cite{Z87} \textsc{V.V. Zhikov} studied minimizers of such functionals as models of anisotropic materials. In certain models, the assumption of a continuous anisotropy change is reasonable, for example in electrorheological fluid dynamics where the anisotropy depends on the smooth electrical field, see \cite{ruzicka}. This is modeled by functionals with a variable-exponent growth, namely $F(x,t) \approx t^{p(x)}$ where the function $p(\cdot)$ is continuous. Other models, such as composite materials, are instead described better by clear-cut transitions and, to this end, \textsc{M.~Colombo \& G.~Mingione} in \cite{CM15_2} developed the regularity theory of the double-phase functional (see also \cite{BCM15_1, BCM15_2, BCM18, CM15_1}). Indeed, the integrand defined as in \eqref{eq:integrand} has the property that the growth rate changes harshly from $p$ to $q$ depending on whether the modulating coefficient $a$ defined as in \eqref{as:a} is zero. From there, an extensive literature has been produced both regarding the double-phase functional, see \cite{BCO18, BO17, BO20, BRS18, BY18}, and integrands with non-standard growth, see for example \cite{BS24, CMMP23, EMMP21, EP24} for results regarding Lipschitz regularity, \cite{HHT17, HO22, K21, Ok2017} for higher regularity results, \cite{O18} for partial H\"older regularity of elliptic systems, and \cite{BR20, BR21, G19} for the case of obstacle problems with non-standard growth.
In particular, it turned out that the condition
\begin{equation}
	\frac{q}{p} \leq 1 + \frac{\alpha}{n}.
	\label{eq:gap-exponents}
\end{equation}
on the gap between the exponents $p$ and $q$ is crucial for the Hölder continuity of minimizers of the double-phase functional and regularity results concerning their gradients, see \cite{BCM18,ELM04,FMM04}.
Concerning global higher integrability, for the case of variable exponents we refer to \cite{EHL11, EHL13}.
In the more delicate case of the double-phase integrand, \textsc{S.-S.~Byun \& J.~Oh} in \cite{BO17} were able to deal with Reifenberg flat domains.
Further, based on the local result \cite{HHK18}, \textsc{A.~Karppinen} in \cite{Karp21} obtained a global higher integrability result for an obstacle problem under the assumption of a measure density condition on the boundary. 
However, the question remains whether it can be generalized to a capacity density condition involving a notion of capacity that is adapted to the double-phase integrand.
Note that Orlicz space capacities in the $x$-independent case go back at least as far as \cite{AB94}, see also \cite{AH03,K00}, and relative capacities for generalized $\Phi$-functions such as the double-phase integrand are already in use in the literature, see \cite{BHH18}, and were inspired by the definitions of capacities for variable exponents, see \cite{DHHR11}.
A partial answer to the main question above and related observations regarding such a notion of capacity are given in the present paper.

Indeed, in the special case of the $p$-energy, a key step toward global higher integrability is the self-improving property of the uniform $p$-fatness condition.
In the Euclidean case, it follows by setting $\alpha=1$ in the more general result in \cite{L88} concerning $(\alpha, p)$-fatness conditions related to Riesz capacities.
Another proof for the self-improvement of uniform $p$-fatness in (weighted) $\R^n$ was given in \cite{Mikkonen}, and more general metric spaces were treated in \cite{BMS01} and \cite{LTV17}. For the case of variable exponents, we refer to \cite{Mikkonen} and \cite{EHL11}, while in the case of certain $\log$-scale distortions of $L^n$, an analogous result was proven in \cite{BK04}.
In particular, we point out that the proof by \textsc{J. Lehrbäck, H. Tuominen \& A. Vähäkangas} \cite{LTV17} relies on the equivalence of uniform $p$-fatness with boundary Poincaré inequalities and pointwise Hardy inequalities, and the corresponding self-improving property of the boundary Poincaré inequalities; for an adaptation to the Euclidean setting we refer to the monograph \cite{KLV21} by \textsc{J. Kinnunen, J. Lehrbäck \& A. Vähäkangas}.
Therefore, we became interested in the connection between these types of inequalities for the double-phase integrand and related capacity density conditions on the complement of $\Omega$.
In this context, a Maz'ya type inequality is a crucial intermediate step.
The expected statement is
$$
	\bint_{B(z,r)} \varphi(x,|u|) \,\dx
	\leq
	\frac{C}{\cp_\varphi\left(\{u=0\} \cap \overline{B\left(z,\frac{r}{2}\right)}, B(z,r)\right)}
	\int_{B(z,r)} \varphi(x,|\nabla u|) \,\dx
$$
with a notion of capacity $\cp_\varphi$ that is adapted to the double-phase integrand.
Somewhat surprisingly we found a counterexample that shows that the preceding formula does not hold with the notion of capacity introduced in \cite{BHH18}, see Example \ref{rem:counterexample-Mazya}.

This led us to introduce the concepts of infimal variational $\varphi$-capacity (see Definition \ref{def_var_cap}) and infimal variational $\varphi$-fatness (see Definition \ref{def:inf-capacity-density}) of sets.
They are based on the notions of level-$t$- and infimal capacity used by \textsc{S.~M.~Buckley \& P.~Koskela} in \cite{BK04} in the $x$-independent case.
First, for a general open set $O \subset \R^n$, we set
$$
    a_O^- := \inf_{x \in O} a(x)
    \quad \text{and} \quad
    a_O^+ := \sup_{x \in O} a(x)
$$
and we let
$$
    \varphi_O^-(t) := t^p  + a_O^- t^q
    \quad \text{and} \quad
    \varphi_O^+(t) := t^p  + a_O^+ t^q.
$$
For an open ball $O=B(z,r)$ with center $z \in \R^n$ and radius $r>0$, we will write $a_{z,r}^\pm := a_{B(z,r)}^\pm$ and $\varphi_{z,r}^\pm := \varphi_{B(z,r)}^\pm$.
In the $x$-independent case $a \equiv a_o \geq 0$ the following definition coincides with the respective notion in \cite{BK04}.

\begin{definition}
\label{def_var_cap}
    Let $E$ be a compact subset of an open set $O \subset \R^n$ and consider a level $t>0$. Then the variational \emph{level-$t$-capacity} associated with $\varphi$ is given by
    $$
        \cp_\varphi^t(E;O)
        :=
        \inf \left\{
        \int_O \frac{\varphi(x,|\nabla u|)}{\varphi_O^-(t)} \,\dx  :
        u \in \Lip(\overline{O}),
        u \geq t\chi_E,
        \left.u\right|_{\partial O} = 0
        \right\}
    $$
    and the \emph{infimal $\varphi$-capacity} of $E$ in $\Omega$ is given by
    $$
        \cp_\varphi^{\inf}(E;O)
        :=
        \inf_{t>0} \cp_\varphi^t(E;O).
    $$
\end{definition}

At this stage, we define local infimal $\varphi$-fatness analogously to \cite{BK04}. Note that in the definition we use the auxiliary functions $\widetilde{\varphi}_{z,r}$, but later in Lemma \ref{rem:formulations-fatness}, we prove that the following definition is equivalent to the standard formulation of fatness given in \eqref{eq:KLV-capacity-density}.

\begin{definition}
\label{def:inf-capacity-density}
A closed set $E \subset \R^n$ is called \emph{locally infimally $\varphi$-fat} at a point $z\in E$ if there exist a constant $c_{\mathrm{fat}}$ and a radius $0<r_o \le 1$ such that
\begin{equation}
    \cp_{\widetilde{\varphi}_{z,r}}^{\inf}\Big(\Big\{x \in \overline{B\big(0,\tfrac{1}{2}\big)} : z+rx \in E\Big\}; B(0,1)\Big)
    \geq
    c_{\mathrm{fat}}
\label{eq:inf-capacity-density}
\end{equation}
for any radius $0<r<r_o$, where $\widetilde{\varphi}_{z,r} \colon B(0,1) \times [0,+\infty) \to [0,+\infty)$ is defined by
\begin{equation}
\label{aux_phi}
    \widetilde{\varphi}_{z,r}(x,t) := \varphi(z+rx,t) = t^p + a(z+rx) t^q.
\end{equation}
If there is a set $S\subset E$ such that $E$ is locally infimally $\varphi$-fat at all $z\in S$ with the same parameters $c_{\mathrm{fat}}$ and $r_o$, we call $E$ locally uniformly infimally $\varphi$-fat at $S$. Whenever $S=E$, we simply refer to $E$ as \emph{locally uniformly infimally $\varphi$-fat}. Furthermore, we use the notation \emph{$p$-fat} instead of infimally $\varphi$-fat whenever $\varphi=|\cdot|^p$.
\end{definition}

In Section \ref{sec:self-improv} we prove that the local infimal $\varphi$-fatness condition is self-improving.
Further, our first main result allows to use the fatness condition on the complement of the domain $\Omega$ in order to prove a Hardy inequality in integral form. The special case of this result for power functions was first proven for $p = 2$ in \cite{A86}, and it was generalized for $p \ge 1$ in \cite{L88, W90}. For Lipschitz domains, there are significantly more precise results, see for example \cite{C99}. Moreover, in \cite{BK04,BHS13} similar results were proven in the case of certain log-scale distortions of $L^n$. 
 
Throughout the paper, $\Lip_0(\Omega)$ denotes the space of Lipschitz functions $u \in \Lip(\Omega)$ with compact support $\spt(u) \Subset \Omega$.
Moreover, as usual, in the following we will denote the Hölder constant of $a$ by
\begin{equation*}
    [a]_{0;\alpha;O} := \sup_{x,y \in O, x \ne y} \frac{|a(x) - a(y)|}{|x-y|^\alpha}, \quad \textnormal{and} \quad [a]_{0;\alpha} := [a]_{0;\alpha;\R^n}.
\end{equation*}

\begin{theorem}[Integral Hardy Inequality]
\label{thm:int_har_ineq}
Let $\varphi$ be defined according to \eqref{eq:integrand} such that \eqref{as:a} and \eqref{eq:gap-exponents} hold with $1<p<q<+\infty$ and assume that $\Omega \subset \R^n$ is a bounded open set such that its complement $\R^n \setminus \Omega$ is locally uniformly infimally $\varphi$-fat with a radius $0<r_o \leq 1$ and fatness constant $c_\mathrm{fat} > 0$.
Then, there exists a constant $c = c(n,p,q,r_o,c_{\mathrm{fat}})$ such that the Hardy inequality in integral form
\begin{equation*}
\begin{aligned}
	\int_\Omega \varphi&\left(x,\frac{|u(x)|}{\dist(x,\partial\Omega)}\right) \,\dx \\
	\leq
	& \ c \left(1+ \diam(\Omega)^{\alpha + \frac{n(p-q)}{p}} [a]_{0;\alpha} \|\nabla u\|_{L^p(\Omega)}^{q-p} \right)^\frac{p}{p-1}
	\int_\Omega \varphi(x,|\nabla u(x)|) \,\dx
\end{aligned}
\end{equation*}
holds true for any $u \in \Lip_0(\Omega)$.
\qed
\end{theorem}

In order to state our second main result, we will denote the integral average of a function $u$ in a set $E$ by
$$
	(u)_E := \bint_E u \,\dx = \frac{1}{|E|} \int_E u \,\dx,
$$
and for an exponent $1 < m < +\infty$ we generalize this notation as
\begin{equation*}
    (u)_{E,m} := \left(\bint_E |u|^m \,\dx \right)^\frac{1}{m}.
\end{equation*}
Moreover, we need to introduce the restricted centered maximal operator, see \cite[Definition 1.27]{KLV21}.

\begin{definition}
\label{def:max_op}
Let $R \colon \R^n \to [0,+\infty]$ be a function. Then the \emph{restricted centered maximal operator} $M_R$ is defined as follows: for any $f \in L^1_\mathrm{loc}(\R^n)$, we have that $M_R f \colon \R^n \to [0,+\infty]$ is given by
\begin{equation*}
    M_R f(x)
    :=
    \begin{dcases}
        \quad |f(x)| & \text{if } R(x) =0, \\
        \sup_{0<r<R(x)} \bint_{B(x,r)} |f(y)| \,\dy & \text{if } R(x) >0.
    \end{dcases}
\end{equation*}
\end{definition}

Our second main result shows that the local $p$-fatness condition, the local infimal $\varphi$-fatness condition \eqref{eq:inf-capacity-density}, a boundary Poincaré inequality, a pointwise Hardy inequality, and a mean value version of the pointwise Hardy inequality are all equivalent; we also refer to Remark \ref{rmk:phases_equiv} for further clarification. In particular, it is surprising that the boundary Poincaré inequality and the pointwise Hardy inequalities, as stated in Theorem \ref{thm-equivalence} and Remark \ref{rmk:phases_equiv}, are already equivalent to the local $p$-fatness condition.

Pointwise Hardy inequalities were introduced in \cite{H99} and \cite{KM97}, where it was already shown that uniform $p$-fatness of the complement guarantees that the domain admits the pointwise $p$-Hardy inequality. In \cite{KLT11} the converse, and thus the equivalence of the two was proven, see also \cite[Theorem 6.23]{KLV21}. In contrast, already in the case of power functions, the integral version of the Hardy inequality is only equivalent to the local uniform $p$-fatness condition if $p = n$.

\begin{theorem}
\label{thm-equivalence}
    Let $\varphi$ be defined according to \eqref{eq:integrand} such that \eqref{as:a} and \eqref{eq:gap-exponents} hold with $1<p<q<+\infty$, assume that $\Omega\subset \R^n$ is an open set, and let $0<r_o \le 1$ and $c_\mathrm{fat}>0$. Then, the following conditions are equivalent:
    \begin{enumerate}
        \item \label{thm:equivalence-0}
        $\R^n \setminus \Omega$ is locally $p$-fat.
        \item \label{thm:equivalence-1}
        $\R^n \setminus \Omega$ is locally uniformly infimally $\varphi$-fat with constants $r_o$ and $c_\mathrm{fat}$.
        \item \label{thm:equivalence-2}
        For any $z\in \R^n \setminus \Omega$, $y\in \R^n$, $0<r < r_o$ and $R>0$ and any $u\in \Lip_0(\Omega)$ we have that
        \begin{equation*}
            \bint_{B(z,r)} \varphi^\pm_{y,R}(|u|) \, \dx
            \leq
            c(n,p,q,c_\mathrm{fat}) \big( 1 + a^\pm_{y,R} (r\nabla u)_{B(z,r),p}^{q-p} \big)
            \bint_{B(z,r)} |r \nabla u|^p \,\dx.
        \end{equation*}
        \item \label{thm:equivalence-3}
        For any $z\in \Omega$ such that $r_z:=\dist(z,\partial\Omega)<r_o$, $y\in \R^n$, $R>0$ and $u\in \Lip_0(\Omega)$, we have that
        \begin{equation*}
             \varphi^\pm_{y,R}\big((u)_{B(z,r_z)} \big)
             \leq
             c(n,p,q,c_\mathrm{fat}) \big( 1 + a^\pm_{y,R} (r_z \nabla u)_{B(z,2r_z),p}^{q-p} \big)
            \bint_{B(z,2r_z)} |r_z \nabla u|^p \,\dx.
        \end{equation*}
        \item \label{thm:equivalence-4}
        For any $z\in \Omega$ such that $r_z:=\dist(z,\partial\Omega)<r_o$, $y\in \R^n$, $R>0$ and $u\in \Lip_0(\Omega)$, we have the pointwise Hardy inequality
        \begin{equation*}
            \varphi^\pm_{y,R}\big( u(z) \big)
            \leq
            c(n,p,q,c_\mathrm{fat}) \left[ 1 + a^\pm_{y,R} M_{2r_z}(|r_z\nabla u|^p )^{\frac{q-p}{p}}(z) \right] M_{2r_z}(|r_z\nabla u|^p )(z).
        \end{equation*}
    \end{enumerate}
    \qed
\end{theorem}

Since the local uniform infimal $\varphi$-fatness is self-improving, see Section \ref{sec:self-improv}, at this stage it would be possible to prove global higher integrability for minimizers of the double-phase functional in $\Omega$ under this condition on $\R^n \setminus \Omega$.
However, we are able to deal with a weaker condition based on a combination of $p$- and $q$-fatness conditions for appropriate subsets of $\R^n \setminus \Omega$.
While it is probably still not optimal, our method is flexible enough to consider quasi-minimizers of general functionals of double-phase type
\begin{equation}
\label{Functional}
    \mathcal{F}\colon W^{1,p}(\Omega) \to [0,+\infty], \qquad \mathcal{F}[u;\Omega] := \int_\Omega F(x,u,\nabla u) \, \dx,
\end{equation}
where the integrand $F \colon \Omega\times \R \times \R^n \to \R$ is a Carathéodory function that satisfies the inequality 
\begin{equation}\label{Functional-growth}
    \nu \, \varphi(x,\xi) \leq F(x,u,\xi) \leq L\, \varphi(x,\xi)
    \qquad \forall \, (x,u,\xi)\in \Omega\times \R \times \R^n
\end{equation}
with constants $0<\nu\leq L$.
Instead of \eqref{eq:gap-exponents}, we impose the strict inequality
\begin{equation}
\label{eq:gap-exponents-strict}
    \frac{q}{p} < 1 + \frac{\alpha}{n}.
\end{equation}

\begin{remark}
\label{rmk:orlicz}
    In order for the class of functionals given by \eqref{Functional} to take only finite values, the right class of functions to consider is the Orlicz-Sobolev space, see \cite[Definition 3]{BHH18},
    \begin{equation*}
        W^{1,\varphi(\cdot)}(\Omega) := \left\{ f \in W^{1,1}_{\loc}(\Omega) \ : \ f, |\nabla f| \in L^{\varphi(\cdot)}(\Omega) \right\}.
    \end{equation*}
    The space $L^{\varphi(\cdot)}(\Omega)$ is called {\it generalized Orlicz space}, also known as Musielak--Orlicz space, see \cite[Definition 2]{BHH18}, and it is given by
    \begin{equation*}
        L^{\varphi(\cdot)}(\Omega) := \left\{ f \in L^0(\Omega) \ : \ \lim_{\lambda \to 0^+} \rho_{\varphi(\cdot)}(\lambda f) = 0 \right\},
    \end{equation*}
    where $L^0(\Omega)$ is the set of measurable functions $f \colon \Omega \to [-\infty, +\infty]$ and
    \begin{equation*}
        \rho_{\varphi(\cdot)}(f) := \int_\Omega \varphi(x, |f(x)|) \, \dx
    \end{equation*}
    is the modular for any chosen $f \in L^0(\Omega)$. Note that in general these spaces can be defined for a wide class of functions, called generalized $\Phi$-functions, and the double-phase integrand \eqref{eq:integrand} is one of them. In particular, for the double-phase function, the space $L^{\varphi(\cdot)}(\Omega)$ is a Banach space, see \cite[Theorem 2.3.13]{DHHR11}.
\end{remark}

Next, we define quasi-minimizers of $\mathcal{F}$ analogously to \cite[Definition 2.1]{Ok2017}.

\begin{definition}
\label{def:quasi-minimizer}
Assume that the functional $\mathcal{F}$ is given by \eqref{Functional} with an integrand satisfying \eqref{Functional-growth}.
A function $u \in W^{1,p}_\mathrm{loc}(\Omega)$ is a quasi-minimizer, or $C_Q$-minimizer, of $\mathcal{F}$ if $F(x,u,\nabla u) \in L^1_\mathrm{loc}(\Omega)$ and there exists a constant $C_Q>0$ such that for any $v \in W^{1,p}(\Omega)$ with $\spt(u-v) \subset \Omega$ there holds
\begin{equation}
    \mathcal{F}[u;\spt(u-v)]
    \leq
    C_Q \mathcal{F}[v;\spt(u-v)].
    \label{eq:quasi-minimizer}
\end{equation}
If $C_Q=1$, we say that $u$ is a minimizer of $\mathcal{F}$.
\end{definition}

We are now ready to state the global higher integrability result. In particular, our fatness assumption on a neighborhood of a boundary point $x$ depends on whether $a$ vanishes in $x$. If this is the case, we assume that $\R^n \setminus \Omega$ is locally uniformly $p$-fat around that point. Otherwise, $q$-fatness is sufficient. Such an assumption ensures that both fatness conditions are self-improving.
For a discussion of the optimality of the condition on $\Omega$, we refer to Remark \ref{rem:optimality}

\begin{theorem}[Global Higher Integrability]
\label{thm:higher_int}
Let $\varphi$ be defined according to \eqref{eq:integrand} such that \eqref{as:a} and \eqref{eq:gap-exponents-strict} hold with $1<p<q<+\infty$.
Let $\Omega \subset \R^n$ be a bounded open set such that $\R^n \setminus \Omega$ is locally uniformly $q$-fat with parameters $c_\mathrm{fat}>0$ and $r_o \in (0,1]$ in the sense of Definition \ref{def:inf-capacity-density}.
Moreover, if there is a point $x \in \partial\Omega$ such that $a(x)=0$, we additionally assume that there exists $r_a>0$ such that $\R^n \setminus \Omega$ is locally uniformly $p$-fat with parameters $c_\mathrm{fat}$ and $\hat{r}_o \in (0,\min\{r_o,r_a\}]$ at
\begin{equation*}
    P(r_a) := (\{x\in \partial\Omega \ : \ a(x)=0\} + B(0,r_a))\setminus \Omega.
\end{equation*}
Further, assume that the integrand $F$ of the functional $\mathcal{F}$ given by \eqref{Functional} satisfies \eqref{Functional-growth}, and that
\begin{equation*}
    u \in f + W^{1,\varphi(\cdot)}_0(\Omega)
    \qquad\text{with }\quad
    f\in W^{1,\varphi^{\sigma_0}(\cdot)}(\Omega),
\end{equation*}
where $\sigma_0>1$, is a $C_Q$-minimizer of $\mathcal{F}$ in the sense of Definition \ref{def:quasi-minimizer}.
Then, there exist $1<\sigma\leq \sigma_0$ and $c\geq 1$, both depending on the data $n,p,q,\alpha,[a]_{0,\alpha},\nu,L,c_\mathrm{fat},C_Q,\sigma_0$ and $\| \nabla u \|_{L^{p}(\Omega,\R^n)} + \| \nabla f \|_{L^{p}(\Omega,\R^n)}$, such that 
    \begin{align}\label{thm:higher_int-estimate}
        \bigg(\int_\Omega &\varphi^\sigma(x,|\nabla u| ) \, \dx \bigg)^{\frac{1}{\sigma}}
        \nonumber\\
        &\leq 
        c \left( 1 + \frac{\diam(\Omega)}{\varrho}\right)^n \Bigg[ \varrho^{\frac{(1-\sigma)n}{\sigma}} \int_{\Omega}  \varphi(x,|\nabla u|) \,\dx
        +\bigg( \int_{\Omega} \varphi^\sigma(x,|\nabla f|) \, \dx \bigg)^{\frac{1}{\sigma}}
        +
        a_\mathcal{P}^{-\frac{p}{q-p}} \Bigg],
    \end{align}
    where
    \begin{equation*}
        a_\mathcal{P}:= \begin{cases}
            \inf\limits_{x\in \partial\Omega\setminus P(r_a/2)} a(x) & \text{if }\partial\Omega\setminus P(r_a/2)\neq \emptyset,\\
            +\infty &\text{otherwise}
        \end{cases}
        \quad\text{and}\quad
        \varrho:= \min\left\{\hat{r}_0, \left(\frac{a_\mathcal{P}}{1+[a]_{0,\alpha}}\right)^{1/\alpha} \right\},
    \end{equation*}
    and we set $\hat{r}_0 = r_0$ provided $a>0$ on $\partial\Omega$. We use the convention $1/\infty = 0$.
    \qed
\end{theorem}

\begin{remark}
    Note that the dependence of the integrability exponent of Theorem \ref{thm:higher_int} on the gradient of $u$ is expected for double-phase equations, since it already occurs in the interior; see for instance \cite[Lemma 2.14]{Ok2017}.
    Further, tracking the constants in the proof of Gehring's lemma, i.e.~Lemma \ref{Gehring's Lemma}, shows that we may choose
    \begin{equation*}
        \sigma=\min\left\{ \sigma_0 , 1 + \frac{1}{c\left(1+[a]_{0,\alpha}\big( \| \nabla u \|_{L^{p}(\Omega,\R^n)} + \| \nabla f \|_{L^{p}(\Omega,\R^n)}\big)^{\beta (q-p)}\right)} \right\}
    \end{equation*}
    for some constant $c=c(n,p,q,\alpha,\nu,L,c_\mathrm{fat},C_Q)\geq 1$ and $\beta=\beta(n,p,q,\alpha)\in (0,1)$.
    In particular, if $a$ is constant, then $\sigma$ does not depend on $\nabla u$.
\end{remark}

As a topic for future research, it would be interesting to investigate the notion of infimal $\varphi$-capacity for more general generalized $\Phi$-functions $\varphi$, see Remark \ref{rmk:orlicz}, and to extend our main results to this setting.
\\
The paper is organized as follows. In Section \ref{capacities} we prove some preliminary results regarding the infimal $\varphi$-capacity and the infimal $\varphi$-fatness condition. In particular, in Theorem \ref{thm:equivalence-p-fatness}, we show that the infimal $\varphi$-capacity density condition and the $p$-capacity density condition are equivalent. In Section \ref{sec:self-improv} we prove the self-improving property of this infimal $\varphi$-fatness, while in Section \ref{sec:aux} we collect auxiliary results already known in the literature that we need in the proofs of our main theorems. In Section \ref{sec:Mazya-type-ineq} we prove a Maz’ya type inequality for certain functionals with $(p, q)$-growth that will be necessary in the proofs of the first two main theorems.
The proof of Theorem \ref{thm:int_har_ineq} is given in Section \ref{sec:int-hardy}, and follows the overall strategy of the proof of \cite[Theorem 2.4]{BK04}.
Next, in Section \ref{proof_thm_eq}, we establish the remaining implications in Theorem \ref{thm-equivalence}.
To this end, we proceed as in \cite[Theorems 6.22 \& 6.23]{KLV21}.
At the end of the paper, we are concerned the application mentioned above, namely in Section \ref{sec:higher_int}, we prove the global higher integrability result in Theorem \ref{thm:higher_int}.


\section{Preliminaries on infimal capacity and capacity density condition}
\label{capacities}

In this section we prove some preliminary results regarding the infimal $\varphi$-capacity and the infimal $\varphi$-fatness condition, with a particular focus on the equivalence between the different definitions of local infimal $\varphi$-fatness and local $p$-fatness.
\\
First, the following truncation lemma shows that it is enough to consider test functions with values in $[0,t]$ in the definition of the level-$t$-capacity.
\begin{lemma}
\label{lem:truncation}
Let $E$ be a compact subset of an open set $O \subset \R^n$ and $t>0$.
Then, we have that
$$
    \cp_\varphi^t(E;O)
    =
    \inf \left\{
    \int_O \frac{\varphi(x,|\nabla u|)}{\varphi^-_O(t)} \,\dx  :
    u \in \Lip(\overline{O}, [0,t]),
    u \geq t\chi_E,
    \left.u\right|_{\partial O} = 0
    \right\}.
$$
\end{lemma}
\begin{proof}
On the one hand, by definition of $\cp_\varphi^t(E;O)$ it is clear that
\begin{align*}
    \cp_\varphi^t(E;O)
    \leq
    \inf \left\{
    \int_O \frac{\varphi(x,|\nabla u|)}{\varphi^-_O(t)} \,\dx  :
    u \in \Lip(\overline{O}, [0,t]),
    u \geq t\chi_E,
    \left.u\right|_{\partial O} = 0
    \right\}.
\end{align*}
On the other hand, assume that the family of admissible test functions for $\cp_\varphi^t(E;O)$, i.e.
$$
    \mathcal{A}^t_{E;O} :=
    \left\{
    u \in \Lip(\overline{O}),
    u \geq t\chi_E,
    \left.u\right|_{\partial O} = 0
    \right\},
$$
is not empty, fix $\varepsilon>0$ and choose $u_\varepsilon \in \mathcal{A}^t_{E;O}$ such that
$$
    \int_O \frac{\varphi(x,|\nabla u_\varepsilon|)}{\varphi^-_O(t)} \,\dx
    \leq
    \cp_\varphi^t(E;O) + \varepsilon.
$$
Observe that $u_{t,\varepsilon} := \min\{u_\varepsilon,t\} \in \Lip(\overline{O},[0,t])$ with $u_{t,\varepsilon} \geq t\chi_E$, $\left.u_{t,\varepsilon}\right|_{\partial O} = 0$ and $|\nabla u_{t,\varepsilon}| \leq |\nabla u_\varepsilon|$ (see for instance \cite[Theorem 2.25]{KLV21} for this property of truncated Sobolev functions).
Together with the fact that $s \mapsto \varphi(x,s)$ is increasing, this implies that
\begin{align*}
    \inf \left\{
    \int_O \frac{\varphi(x,|\nabla u|)}{\varphi^-_O(t)} \,\dx  :
    u \in \Lip(\overline{O}, [0,t]),
    u \geq t\chi_E,
    \left.u\right|_{\partial O} = 0
    \right\}
    &\leq
    \int_O \frac{\varphi(x,|\nabla u_{t,\varepsilon}|)}{\varphi^-_O(t)} \,\dx
    \\ &\leq
    \int_O \frac{\varphi(x,|\nabla u_\varepsilon|)}{\varphi^-_O(t)} \,\dx
    \\& \leq
    \cp_\varphi^t(E;O) + \varepsilon.
\end{align*}
Since $\varepsilon>0$ was arbitrary, this concludes the proof of the lemma.
\end{proof}

Next, we prove the following upper bound for the infimal $\varphi$-capacity.
\begin{lemma}
\label{lem:upper-estimate-infcap}
For any compact subset $E$ of an open set $O \subset \R^n$ and for any $1 < p< +\infty$, we have that
\begin{equation*}
	\cp_\varphi^{\inf}(E;O)
    \leq
    \cp_p(E;O).
\end{equation*}
\end{lemma}
\begin{proof}
Note that $u$ is an admissible comparison function for $\cp_\varphi^1(E;O)$ if and only if $tu$ is an admissible comparison function for $\cp_\varphi^t(E;O)$, with $t > 0$. Thus, for an admissible function $u$ for $\cp_\varphi^1(E;O)$, we find that
\begin{align*}
	\int_O \frac{\varphi(x,|\nabla (tu)|)}{\varphi_O^-(t)} \,\dx
    &=
    \int_O \frac{t^p|\nabla u|^p + a t^q |\nabla u|^q}{t^p+ a_O^- t^q} \,\dx
    =
    \int_O \frac{|\nabla u|^p + a t^{q-p} |\nabla u|^q}{1+a_O^- t^{q-p}} \,\dx\\
    &\leq
    \int_O |\nabla u|^p \,\dx + t^{q-p} \int_O a|\nabla u|^q \,\dx.
\end{align*}
Hence, for any admissible $u$ for $\cp_\varphi^1(E;O)$, we conclude that
\begin{align*}
    \cp_\varphi^{\inf}(E;O)
    &=
    \inf_{t>0} \cp_\varphi^t(E;O) \\
    &\leq
    \inf_{t>0} \int_O \frac{\varphi(x,|\nabla (tu)|)}{\varphi_O^-(t)} \,\dx \\
    &\leq
    \inf_{t>0} \left(
    \int_O |\nabla u|^p \,\dx + t^{q-p} \int_O a |\nabla u|^q \,\dx
    \right) \\
    &=
    \int_O |\nabla u|^p \,\dx.
\end{align*}
Taking the infimum over all admissible test functions $u$ for $\cp_\varphi^1(E;O)$ (note that the left-hand side of the preceding inequality does not depend on $u$), we obtain the claim of the lemma.
\end{proof}

Further, we have the following lower bound.
\begin{lemma}
\label{lem:lower-estimate-infcap}
Let $E$ be a compact subset of a bounded, open set $O \subset \R^n$ and let $1 < p< q < +\infty$.
Then we have that
\begin{equation*}
	\cp_\varphi^{\inf}(E;O)
	\geq
	c(p,q)
	\min\left\{ \cp_p(E;O), |O| \left( \frac{\cp_p(E;O)}{|O|} \right)^\frac{q}{p} \right\}.
\end{equation*}
\end{lemma}
\begin{proof}
Assume that $u \in \Lip(\overline{O})$ with $u=0$ on $\partial O$ and $u \geq \chi_E$.
Since for any fixed $x \in O$ and $t>0$
$$
	\frac{1}{1+a(x)t^{q-p}} |\nabla u|^p + \frac{a(x) t^{q-p}}{1+ a(x) t^{q-p}} |\nabla u|^q
$$
is a convex combination of $|\nabla u|^p$ and $|\nabla u|^q$, we obtain that
\begin{align}
	\int_{O} \frac{\varphi(x,|\nabla (tu)|)}{\varphi_O^-(t)} \,\dx
    &\geq
    \int_{O} \frac{t^p|\nabla u|^p + a t^q |\nabla u|^q}{t^p+ a t^q} \,\dx
    \nonumber \\
    &=
    \int_{O} \left[ \frac{1}{1+at^{q-p}} |\nabla u|^p + \frac{a t^{q-p}}{1+ a t^{q-p}} |\nabla u|^q \right] \dx
    \nonumber \\
    &\geq
    \int_{O} \min\{ |\nabla u|^p, |\nabla u|^q \} \,\dx
    \nonumber \\
    &=
    \int_{O} \left[ |\nabla u|^p \chi_{\{|\nabla u| \geq 1\}}
    + |\nabla u|^q \chi_{\{|\nabla u| < 1\}} \right] \dx.
    \label{eq:lower-bound-aux1}
\end{align}
By Hölder's inequality, we find that
$$
	\int_{O} |\nabla u|^q \chi_{\{|\nabla u| < 1\}} \,\dx
	\geq
	|O|^{1-\frac{q}{p}}
	\left( \int_{O} |\nabla u|^p \chi_{\{|\nabla u| < 1\}} \,\dx \right)^\frac{q}{p}.
$$
Inserting this into the right-hand side of \eqref{eq:lower-bound-aux1}, we infer
\begin{align*}
	\int_{O} \frac{\varphi(x,|\nabla (tu)|)}{\varphi_O^-(t)} \,\dx
    &\geq
    \frac{\left( \int_{O} |\nabla u|^p \chi_{\{|\nabla u| \geq 1\}} \,\dx \right)^\frac{q}{p}}{\left( \int_{O} |\nabla u|^p \chi_{\{|\nabla u| \geq 1\}} \,\dx \right)^{\frac{q}{p}-1}}
    + \frac{\left( \int_{O} |\nabla u|^p \chi_{\{|\nabla u| < 1\}} \,\dx \right)^\frac{q}{p}}{|O|^{\frac{q}{p}-1}} \\
    &\geq
    \frac{c(p,q) \left( \int_{O} |\nabla u|^p \,\dx \right)^\frac{q}{p}}{\max\left\{ \left( \int_{O} |\nabla u|^p \chi_{\{|\nabla u| \geq 1\}} \,\dx \right)^{\frac{q}{p}-1}, |O|^{\frac{q}{p}-1} \right\}} \\
    &\geq
    \frac{c(p,q) \left( \int_{O} |\nabla u|^p \,\dx \right)^\frac{q}{p}}{\max\left\{ \left( \int_{O} |\nabla u|^p \,\dx \right)^{\frac{q}{p}-1}, |O|^{\frac{q}{p}-1} \right\}} \\
    &=
    c(p,q) \min\left\{ \int_{O} |\nabla u|^p \,\dx,
    |O| \left( \frac{\int_{O} |\nabla u|^p \,\dx}{|O|} \right)^\frac{q}{p} \right\}\\
    &\geq
    c(p,q) \min\left\{ \cp_p(E;O),
    |O| \left( \frac{\cp_p(E;O)}{|O|} \right)^\frac{q}{p} \right\}.
\end{align*}
Taking the infimum over all admissible functions in the definition of $\cp_\varphi^t(E;O)$ on the left-hand side of the preceding inequality and thereafter taking the infimum over $t>0$ yields the claim of the lemma. 
\end{proof}

In Definition \ref{def:inf-capacity-density}, we defined the infimal $\varphi$-fatness using the auxiliary functions $\widetilde{\varphi}_{z,r}$.
In \cite{KLV21} on the other hand, a closed set $E \subset \R^n$ satisfies a uniform $p$-capacity density condition if there exists $\hat{c}_\mathrm{fat}>0$ such that
\begin{equation}
    \cp_p\Big(E \cap \overline{B\big(z,\tfrac{r}{2}\big)} ; B(z,r)\Big)
	\geq
	\hat{c}_{\mathrm{fat}}
	\cp_p\Big(\overline{B\big(z,\tfrac{r}{2}\big)} ; B(z,r)\Big)
    \label{eq:p-fatness-formulation-KLV}
\end{equation}
holds for any radius $r>0$ and any point $z \in E$.
However, a simple change of variables shows that 
\begin{equation}
	\cp_p\Big(E \cap \overline{B\big(z,\tfrac{r}{2}\big)} ; B(z,r)\Big)
	=
	r^{n-p} \cp_p\Big(\Big\{x \in \overline{B\big(0,\tfrac{1}{2}\big)} : z+rx \in E\Big\}; B(0,1)\Big).
	\label{eq:equivalence-density-aux}
\end{equation}
Noting that
\begin{equation}
\label{eq:scaling-p-capacity}
	c(n,p) r^{n-p}
	\leq
	\cp_p\Big(\overline{B\big(z,\tfrac{r}{2}\big)} ; B(z,r)\Big)
	\leq
	c(n) r^{n-p}
\end{equation}
(see e.g.~\cite[Lemma 5.35]{KLV21}), we conclude that the definition of $p$-fatness relying on \eqref{eq:p-fatness-formulation-KLV} is equivalent to the notion of $p$-fatness that results from choosing the special case $\varphi(s) = s^p$ in Definition \ref{def:inf-capacity-density}.

In the next lemma, we establish the analogous characterization for infimal $\varphi$-fatness.
To this end, we prove that for any $z \in E$, where $E$ denotes a closed set, \eqref{eq:inf-capacity-density} is equivalent to the condition
\begin{equation}
	\cp_\varphi^{\inf}\Big(E \cap \overline{B\big(z,\tfrac{r}{2}\big)} ; B(z,r)\Big)
	\geq
	\tilde{c}_{\mathrm{fat}}
	\cp_\varphi^{\inf}\Big(\overline{B\big(z,\tfrac{r}{2}\big)} ; B(z,r)\Big)
	\label{eq:KLV-capacity-density}
\end{equation}
for any $0<r<\tilde{r}_o \le 1$.

\begin{lemma}
\label{rem:formulations-fatness}
    Let $E \subset \R^n$ be closed and let $z \in E$. Then \eqref{eq:inf-capacity-density} is satisfied with constants $0 < r_o \le 1$ and $c_{\mathrm{fat}}$ if and only if \eqref{eq:KLV-capacity-density} is satisfied with constants $0 <\tilde{r}_o \le 1$ and $\tilde{c}_{\mathrm{fat}}$. The constants $r_o$ and $\tilde{r}_o$, and $c_{\mathrm{fat}}$ and $\tilde{c}_{\mathrm{fat}}$, respectively, only depend on $n$, $p$, $q$ and each other.
\end{lemma}
\begin{proof}
Let $\varphi(x,s) := s^p + a(x) s^q$, and note that $a_{z,r}^- = \inf_{y \in B(z,r)} a(y) = \inf_{x \in B(0,1)} a(z+rx)$ and thus
$$
    \varphi_{z,r}^-(t)
    =
    \widetilde{\varphi}_{0,1}^-(t) := t^p + \left(\inf_{x \in B(0,1)} a(z+rx)\right) t^q.
$$
Further, we find that $u \in \Lip(B(0,1))$ if and only if $\tilde{u} = u\left(\frac{\cdot - z}{r}\right) \in \Lip(B(z,r))$ with $\left.u\right|_{\partial B(0,1)}=0$ if and only if $\left.\tilde{u}\right|_{\partial B(z,r)}=0$.
Moreover, we have that
$$
	\Big\{x \in \overline{B\big(0,\tfrac{1}{2}\big)} : z+rx \in E\Big\}
	=
	\Big\{\tfrac{y-z}{r} : y \in E \cap \overline{B\big(z,\tfrac{r}{2}\big)}\Big\}
$$
and hence that
$$
	u \geq t\chi_{\Big\{x \in \overline{B\big(0,\frac{1}{2}\big)} : z+rx \in E\Big\}}
	\quad \text{if and only if} \quad
	\tilde{u} \geq t \chi_{E\cap \overline{B\left(z,\frac{r}{2}\right)}}.
$$
Thus, $u$ is an admissible comparison map in the definition of $\cp_{\widetilde{\varphi}}^t\Big(\Big\{x \in \overline{B\big(0,\tfrac{1}{2}\big)} : z+rx \in E\Big\}; B(0,1)\Big)$ if and only if $\tilde{u}$ is an admissible comparison map for $\cp_\varphi^t\left(E\cap \overline{B\left(z,\frac{r}{2}\right)};B(z,r)\right)$.
For any fixed level $t > 0$, we compute that
\begin{align}\label{eq:equivalence-density-aux3}
	\int_{B(z,r)} \frac{\varphi(y,|\nabla \tilde{u}(y)|)}{\varphi_{z,r}^-(t)} \,\dy
	&=
	\int_{B(z,r)} \frac{\varphi\left(y, \left|\nabla \left( \frac{u}{r} \right) \left(\frac{y-z}{r} \right) \right| \right)}{\varphi_{z,r}^-(t)} \,\dy \nonumber\\
	&=
	r^n\int_{B(0,1)} \frac{\varphi\left(z+rx, \left|\nabla \left( \frac{u}{r} \right)(x) \right| \right)}{\widetilde{\varphi}_{0,1}^-(t)} \,\dx \nonumber\\
	&=
	r^n\int_{B(0,1)} \frac{\widetilde{\varphi}_{0,1}^-\left( \frac{t}{r} \right)}{\widetilde{\varphi}_{0,1}^-(t)}
	\cdot \frac{\widetilde{\varphi}_{z,r}\left(x, \left|\nabla \left( \frac{u}{r} \right)(x) \right| \right)}{\widetilde{\varphi}_{0,1}^-\left(\frac{t}{r} \right)} \,\dx \nonumber\\
	&\geq
	r^{n-p} \cp_{\widetilde{\varphi}}^{\frac{t}{r}}
	\Big(\Big\{x \in \overline{B\big(0,\tfrac{1}{2}\big)} : z+rx \in E\Big\}; B(0,1)\Big)\nonumber\\
	&\geq
	r^{n-p} \cp_{\widetilde{\varphi}}^{\inf}\Big(\Big\{x \in \overline{B\big(0,\tfrac{1}{2}\big)} : z+rx \in E\Big\}; B(0,1)\Big),
\end{align}
where in the penultimate line, we used that $r \leq 1$.
Thus, assuming that \eqref{eq:inf-capacity-density} is true, by Lemma \ref{lem:upper-estimate-infcap} and \eqref{eq:scaling-p-capacity} we have that
\begin{align*}
	\int_{B(z,r)} \frac{\varphi(y,|\nabla \tilde{u}(y)|)}{\varphi_{z,r}^-(t)} \,\dy
	&\geq
	c_\mathrm{fat} r^{n-p}
	\geq
	c(n) c_\mathrm{fat} \cp_p\Big(\overline{B\big(z,\tfrac{r}{2}\big)} ; B(z,r)\Big)\\
	&\geq
	c(n) c_\mathrm{fat}
	\cp_\varphi^{\inf}\Big(\overline{B\big(z,\tfrac{r}{2}\big)} ; B(z,r)\Big).
\end{align*}
Taking the infimum over all admissible comparison maps $\tilde{u}$ and all levels $t>0$, this implies \eqref{eq:KLV-capacity-density} with constant $\tilde{c}_\mathrm{fat} := c(n) c_\mathrm{fat}$ and radius $0<\tilde{r}_o \leq 1$.

Next, assuming that \eqref{eq:KLV-capacity-density} holds and using Lemma \ref{lem:upper-estimate-infcap}, we obtain that
\begin{align}
	r^{p-n} \cp_p \Big(E \cap \overline{B\big(z,\tfrac{r}{2}\big)} ; B(z,r)\Big)
	&\geq
	r^{p-n} \cp_\varphi^{\inf} \Big(E \cap \overline{B\big(z,\tfrac{r}{2}\big)} ; B(z,r)\Big)
	\nonumber \\
	&\geq
	\tilde{c}_\mathrm{fat} r^{p-n} \cp_\varphi^{\inf}\Big(\overline{B\big(z,\tfrac{r}{2}\big)} ; B(z,r)\Big)
	\nonumber \\
	&\geq
	c(n,p,q) \tilde{c}_\mathrm{fat}
	\label{eq:equivalence-density-aux2}
\end{align}
for sufficiently small radii $r>0$.
Indeed, in the last line of the preceding inequality, Lemma \ref{lem:lower-estimate-infcap} and \eqref{eq:scaling-p-capacity} yielded
\begin{align*}
	&\cp_\varphi^{\inf}\Big(\overline{B\big(z,\tfrac{r}{2}\big)} ; B(z,r)\Big) \\
	\geq
	& \, c(p,q) \min\left\{
	\cp_p\Big(\overline{B\big(z,\tfrac{r}{2}\big)} ; B(z,r)\Big),
	|B(z,r)| \left( |B(z,r)|^{-1}
	\cp_p\Big(\overline{B\big(z,\tfrac{r}{2}\big)} ; B(z,r)\Big) \right)^\frac{q}{p}	
	\right\}  \\
	\geq
	& \, c(p,q) \min \big\{
	c(n,p) r^{n-p}, c(n,p,q) r^{n-q}
	\big\} \\
	=
	& \, c(n,p,q) r^{n-p}
\end{align*}
provided that the radius $r$ is small enough depending on $n$, $p$ and $q$.
Therefore, applying once again Lemma \ref{lem:lower-estimate-infcap} and then \eqref{eq:equivalence-density-aux} and \eqref{eq:equivalence-density-aux2}, we conclude that
\begin{align*}
	&\cp_{\widetilde{\varphi}_{z,r}}^{\inf}\Big(\Big\{x \in \overline{B\big(0,\tfrac{1}{2}\big)} : z+rx \in E\Big\}; B(0,1)\Big) \\
	\geq
	&\min\Big\{
	\cp_p\Big(\Big\{x \in \overline{B\big(0,\tfrac{1}{2}\big)} : z+rx \in E\Big\}; B(0,1)\Big),\\
	&\qquad\qquad
	|B(0,1)|\left(|B(0,1)|^{-1}
	\cp_p\Big(\Big\{x \in \overline{B\big(0,\tfrac{1}{2}\big)} : z+rx \in E\Big\}; B(0,1)\Big)\right)^\frac{q}{p}
	\Big\} \\
	=
	&\min\Big\{
	r^{p-n}\cp_p\Big(E \cap \overline{B\big(z,\tfrac{r}{2}\big)} ; B(z,r)\Big),
	c(n,p,q)\left(
	r^{p-n}\cp_p\Big(E \cap \overline{B\big(z,\tfrac{r}{2}\big)} ; B(z,r)\Big)\right)^\frac{q}{p}
	\Big\} \\
	\ge
	& \, c(n,p,q,\tilde{c}_\mathrm{fat}).
\end{align*}

Thus, \eqref{eq:KLV-capacity-density} implies \eqref{eq:inf-capacity-density} with a constant $c_\mathrm{fat} := c(n,p,q,\tilde{c}_\mathrm{fat})$ and radius $r_o = r_o(n,p,q,\tilde{r}_o)$.
\end{proof}

Lastly, in the next result we want to show the equivalence of local infimal $\varphi$-fatness and local $p$-fatness.

\begin{theorem}
\label{thm:equivalence-p-fatness}
A closed set $E \subset \R^n$ is locally uniformly infimally $\varphi$-fat in the sense of Definition~\ref{def:inf-capacity-density} if and only if it is locally uniformly $p$-fat.
\end{theorem}
\begin{proof}
On the one hand, assuming that \eqref{eq:inf-capacity-density} holds true, by Lemma \ref{lem:upper-estimate-infcap} we have that
\begin{align*}
    &\cp_p\Big(\Big\{x \in \overline{B\big(0,\tfrac{1}{2}\big)} : z+rx \in E\Big\}; B(0,1)\Big) \\
    \geq
    &\cp_{\widetilde{\varphi}_{z,r}}^{\inf}\Big(\Big\{x \in \overline{B\big(0,\tfrac{1}{2}\big)} : z+rx \in E\Big\}; B(0,1)\Big)
    \geq
	c_{\mathrm{fat}},
\end{align*}
and by Lemma \ref{rem:formulations-fatness} we have the first implication.

On the other hand, for any radius $r>0$ smaller than some $\hat{r}_o = \hat{r}_o(n,p,q) \le 1$, by \eqref{eq:scaling-p-capacity} we obtain that
\begin{equation}
\label{aux-capp-estimates}
\begin{aligned}
    &|B(z,r)|^{1-\frac{q}{p}} \cp_p\left( \overline{B\big(z,\tfrac{r}{2}\big)} ; B(z,r) \right)^\frac{q}{p} \\
    \geq
    & \, c(n,p,q) r^{n-q} \geq
    c(n,p,q) r^{n-p}
    \geq
    c(n,p,q) \cp_p\left( \overline{B\big(z,\tfrac{r}{2}\big)} ; B(z,r) \right).
\end{aligned}
\end{equation}

Therefore, if $E$ is locally uniformly $p$-fat, by Lemma \ref{lem:lower-estimate-infcap} we find that
\begin{align*}
    &\cp_\varphi^{\inf} \Big(E \cap \overline{B\big(z,\tfrac{r}{2}\big)} ; B(z,r)\Big)\\
    \geq
    &\min\left\{
    \cp_p \Big(E \cap \overline{B\big(z,\tfrac{r}{2}\big)} ; B(z,r)\Big),
    |B(z,r)|^{1-\frac{q}{p}} \cp_p\left(E \cap \overline{B\big(z,\tfrac{r}{2}\big)} ; B(z,r) \right)^\frac{q}{p}
    \right\}\\
    \geq
    &\min\left\{
    \hat{c}_\mathrm{fat} \cp_p \Big(\overline{B\big(z,\tfrac{r}{2}\big)} ; B(z,r)\Big),
    \hat{c}_\mathrm{fat}^\frac{q}{p} |B(z,r)|^{1-\frac{q}{p}} \cp_p\left( \overline{B\big(z,\tfrac{r}{2}\big)} ; B(z,r) \right)^\frac{q}{p}
    \right\}\\
    \geq
    & \, \tilde{c}_\mathrm{fat}(n,p,q,\hat{c}_\mathrm{fat}) \cp_p \Big(\overline{B\big(z,\tfrac{r}{2}\big)} ; B(z,r)\Big) \\
    \geq
    & \, \tilde{c}_\mathrm{fat}(n,p,q,\hat{c}_\mathrm{fat}) \cp_\varphi^{\inf} \Big(\overline{B\big(z,\tfrac{r}{2}\big)} ; B(z,r)\Big).
\end{align*}

for any radius $0<r< \hat{r}_o(n,p,q)$, where in the last estimate we used Lemma \ref{lem:upper-estimate-infcap}.
In view of Lemma \ref{rem:formulations-fatness}, this shows the claim of the lemma.
\end{proof}


\section{Self-improving property of the infimal \texorpdfstring{$\varphi$}{phi}-capacity density condition}
\label{sec:self-improv}

In this section we prove that the infimal $\varphi$-fatness condition is self-improving.
We begin with a slightly more precise version of the classical result for the $p$-capacity density condition.
Essentially, we refer to the proof of \cite[Theorem 7.21]{KLV21}, which relies on the equivalence of $p$-fatness of the complement with the boundary Poincaré inequality \cite[Theorems 6.22 \& 6.23]{KLV21}, and the self-improvement of the boundary Poincaré inequality in \cite[Theorem 7.24]{KLV21}, with the minor difference that we track the centers and radii of the balls in these proofs.
While we will not need the improvement in this section, we will use it in the subsequent section when proving the global higher integrability result.

\begin{theorem}
\label{thm:pselfimproving2}
    Let $1<p<+\infty$. Assume that $E\subset \R^n$ is a closed set and that there exists $B(x,r_a)$ with $x \in E$ such that $E$ is locally uniformly $p$-fat in the sense of Definition \ref{def:inf-capacity-density} at $E\cap \overline{B(x,r_a)}$ with some constant $c_\mathrm{fat}$ and $r_0>0$. Then there exist $\varepsilon = \varepsilon(n,p,c_\mathrm{fat})$ with $0<\varepsilon<p-1$ and a constant $\widetilde{c}_\mathrm{fat}=\widetilde{c}_\mathrm{fat}(n,p,c_\mathrm{fat})$ such that $E$ is locally uniformly $(p-\varepsilon)$-fat at $E\cap \overline{B\left(x,\tilde{r}_0\right)}$ with the parameters $\widetilde{c}_\mathrm{fat}$ and $\tilde{r}_0$, where $\tilde{r}_0:=\tfrac{1}{4}\min\left\{r_a,r_0\right\}$.
\end{theorem}
\begin{proof}
    We can assume without loss of generality that $(\R^n\setminus E) \cap  \overline{B\left(x,\min\left\{r_a,r_0\right\}/2\right)} \neq \emptyset $, since otherwise $B(z,\varrho/2)\subset E$ for every $0<\varrho<\min\left\{r_a,r_0\right\}/2$ and $z\in B(x,\min\left\{r_a,r_0\right\}/4)$. Thus there is nothing to prove in this scenario. From the proof of \cite[Theorem 6.22]{KLV21} it is apparent that the local uniform $p$-fatness of $E$ at $E\cap \overline{B(x,r_a)}$ yields the boundary $p$-Poincaré inequality for any test function $u \in \Lip_0\big( \R^n \setminus E \big)$ and ball $B(z,\varrho)$ with $z \in E\cap \overline{B(x,r_a)}$ and $\varrho < r_0$.
    Next, we need to ensure that this inequality is self-improving by tracking the balls in the proof of \cite[Theorem 7.24]{KLV21}, in particular \cite[Lemma 7.29]{KLV21}, where the $p$-fatness condition is applied to smaller balls contained in $B(z,\varrho)$.
    In our setting, the $p$-fatness condition can be used for every ball in $B(z,\varrho)$, $z\in E \cap \overline{B\left(x,r_a/2\right)}$, with parameters $c_\mathrm{fat}$ and $r_0>0$, provided that we have $\varrho \leq \min\left\{\frac{r_a}{2},r_0\right\}$.
    Thus, we obtain the boundary $(p-\epsilon)$-Poincaré inequality for any test function $u \in \Lip_0( \R^n \setminus E )$ and any ball $B(z,\varrho)$ with $z\in E \cap \overline{B\left(x,r_a/2\right)}$ and $\varrho \leq \min\left\{r_a/2,r_0\right\}$.
    Now, let us pay attention to the proof of \cite[Theorem 6.23]{KLV21}, in particular the step from (b) to (c). There, a boundary point $z\in \partial E$ is selected corresponding to some $y\in \R^n \setminus E$ and the boundary Poincaré inequality is applied to the ball $B(z,\dist(y,\partial E))$. In our setting we therefore need that $\dist(y,\partial E)\leq \min\left\{r_a/2,r_0\right\}$ and $z$ to be contained in $\overline{B\left(x,r_a/2\right)}$. This however is the case if $y\in (\R^n\setminus E)\cap \overline{B\left(x,\min\left\{r_a,r_0\right\}/2\right)}$ which is nonempty by assumption. Therefore, we obtain the pointwise $(p-\varepsilon)$-Hardy inequality
    $$
        |u(y)| \leq
        c \dist(y,\partial E) \left( M_{2\dist(y,\partial E)} |\nabla u|^{p-\varepsilon}(y) \right)^\frac{1}{p-\varepsilon}
    $$
    for any $u \in \Lip_0( \R^n \setminus E)$ and any $y \in ( \R^n \setminus E ) \cap \overline{B\left(x,\min\left\{r_a,r_0\right\}/2\right)}$. According to \cite[Theorem 6.23]{KLV21}, namely the step from (c) to (d), this implies the pointwise $(p-\varepsilon)$-Hardy inequality for any $y \in ( \R^n \setminus E ) \cap \overline{B\left(x,\min\left\{r_a,r_0\right\}/2\right)}$. Let us now consider the step from (d) to (a) in the proof of \cite[Theorem 6.23]{KLV21}. There, the authors establish that $E$ is $(p-\varepsilon)$-fat at a point $z\in E$ in the sense of Definition \ref{def:inf-capacity-density} with parameters $\Tilde{r}_0$ and $\widetilde{c}_\mathrm{fat}$, provided that the poinwise $(p-\varepsilon)$-Hardy inequality can be applied to every $y\in (\R^n\setminus E) \cap B(z,\tilde{r}_0/6)$. But if we choose $\Tilde{r}_0 = \min\left\{r_a,r_0\right\}/4$ and assume $z\in E\cap \overline{B\left(x,\Tilde{r}_0\right)}$ this is quite certainly possible due to the previous comments. Therefore, $E$ is locally uniformly $(p-\varepsilon)$-fat at $E\cap \overline{B\left(x,\tilde{r}_0\right)}$ with the constant $\widetilde{c}_\mathrm{fat}$ and $\tilde{r}_0=\tfrac{1}{4}\min\left\{r_a,r_0\right\}$.
\end{proof}

Based on Theorem \ref{thm:pselfimproving2} and the lower bound in Lemma \ref{lem:lower-estimate-infcap}, we are now ready to prove that the infimal $\varphi$-fatness condition is self-improving.

\begin{lemma}
\label{lem:phiselfimproving}
    Assume that $E\subset \R^n$ is a closed set which is locally uniformly infimally $\varphi$-fat in the sense of Definition \ref{def:inf-capacity-density} with some constant $c_\mathrm{fat}$. Then there exists $\varepsilon=\varepsilon(n,p,q,c_\mathrm{fat})$ with $0<\varepsilon<p-1$ and a constant $\hat{c}_\mathrm{fat}=\hat{c}_\mathrm{fat}(n,p,q,c_\mathrm{fat})$ such that $E$ is locally uniformly infimally $\varphi^{1-\frac{\varepsilon}{p}}$-fat for the constant $\hat{c}_\mathrm{fat}$.
\end{lemma}

\begin{proof}
    Let $\delta>0$ to be chosen. We start the proof with the estimates
    \begin{equation}\label{aux-phiselfimproving1}
        \tfrac{1}{2} \left( t^{p(1-\delta)} + a^{1-\delta} t^{q(1-\delta)} \right)
        \leq
        \max\left\{ t^{p(1-\delta)} , a^{1-\delta} t^{q(1-\delta)} \right\}
        \leq
        \left( t^{p} + a t^{q} \right)^{1-\delta}
    \end{equation}
    and
    \begin{equation}\label{aux-phiselfimproving2}
        \left( t^{p} + a t^{q} \right)^{1-\delta}
        \leq
        2^{1-\delta} \max\left\{ t^{p(1-\delta)} , a^{1-\delta} t^{q(1-\delta)} \right\}
        \leq 
        2\left( t^{p(1-\delta)} + a^{1-\delta} t^{q(1-\delta)} \right),
    \end{equation}
    which are valid for every $t\geq 0$ and $a\in \R_{\geq 0}$, in particular for the choices $a=a(x)$ as in \eqref{as:a} and $a=a^-_O$ for any open set $O\subset \R^n$.
    Now, we define the auxiliary function $\widehat{\varphi}$ by
    \begin{equation*}
        \widehat{\varphi}(x,t) = \left( t^{p(1-\delta)} + a^{1-\delta}(x) t^{q(1-\delta)} \right).
    \end{equation*}
    Since the proof of Lemma \ref{lem:lower-estimate-infcap} only relies on the boundedness of $a$, which is also satisfied by $a^{1-\delta}$, we may apply Lemma \ref{lem:lower-estimate-infcap} to $\widehat{\varphi}$ with parameters $(p(1-\delta),q(1-\delta))$ replacing $(p,q)$.
    Furthermore, we observe that Theorem \ref{thm:equivalence-p-fatness} yields the existence of $c^*_\mathrm{fat}>0$ depending only on $n,p,q$ and $c_\mathrm{fat}$ such that $E$ is locally uniformly $p$-fat with fatness constant $c^*_\mathrm{fat}$.
    Now, Theorem \ref{thm:pselfimproving2} allows us to infer the existence of some $\varepsilon=\varepsilon(n,p,c^*_\mathrm{fat}) \in (0,p-1)$ and $\hat{c}^*_\mathrm{fat}=\hat{c}^*_\mathrm{fat}(n,p,c^*_\mathrm{fat}) >0$ such that $E$ is locally uniformly $(p-\varepsilon)$-fat with fatness constant $\hat{c}^*_\mathrm{fat}$, namely that there holds
    \begin{equation*}
        \cp_{p-\varepsilon} \Big(E \cap \overline{B\big(z,\tfrac{r}{2}\big)}; B(z,r)\Big)
        \geq
        \hat{c}^*_\mathrm{fat} \cp_{p-\varepsilon} \Big(\overline{B\big(z,\tfrac{r}{2}\big)} ; B(z,r)\Big).
    \end{equation*}
    Additionally, analogously to \eqref{aux-capp-estimates} (see \cite[Lemma 5.35]{KLV21} in particular), for any radius $r>0$ that is smaller than some $\hat{r}_o = \hat{r}_o(n,p,q,\hat{c}^*_\mathrm{fat})\leq 1$ we have the chain of inequalities 
    \begin{align*}
         \big(\hat{c}^*_\mathrm{fat}\big)^\frac{q}{p} |B(z,r)|^{1-\frac{q}{p}} & \cp_{p-\varepsilon}\left( \overline{B\big(z,\tfrac{r}{2}\big)} ; B(z,r) \right)^\frac{q}{p}
        \geq
        c(n,p,q) \big(\hat{c}^*_\mathrm{fat}\big)^\frac{q}{p}  r^{n-(p-\varepsilon) \frac{q}{p}} 
        \nonumber\\ \qquad &\geq
        c(n,p,q) \big(\hat{c}^*_\mathrm{fat}\big)^\frac{q}{p}  r^{n-(p-\varepsilon)} 
        \geq
        c(n,p,q) \big(\hat{c}^*_\mathrm{fat}\big)^\frac{q}{p} \cp_{p-\varepsilon}\left( \overline{B\big(z,\tfrac{r}{2}\big)} ; B(z,r) \right).
    \end{align*}
    
    Hence, we conclude from Lemma \ref{lem:lower-estimate-infcap}, $\delta=\frac{\varepsilon}{p}$, the penultimate inequality, and Lemma \ref{lem:upper-estimate-infcap} that
    \begin{align*}
        &\cp_{\widehat{\varphi}}^{\inf} \Big(E \cap \overline{B\big(z,\tfrac{r}{2}\big)} ; B(z,r)\Big)
        \\&\geq
        c(p,q) \min\left\{
        \cp_{p-\varepsilon} \Big(E \cap \overline{B\big(z,\tfrac{r}{2}\big)} ; B(z,r)\Big),
        |B(z,r)|^{1-\frac{q}{p}} \cp_{p-\varepsilon}\left( E \cap \overline{B\big(z,\tfrac{r}{2}\big)} ; B(z,r) \right)^\frac{q}{p}
        \right\}
        \\ &\geq
        c(p,q) \min\left\{
        \hat{c}^*_\mathrm{fat} \cp_{p-\varepsilon} \Big(\overline{B\big(z,\tfrac{r}{2}\big)} ; B(z,r)\Big),
        \big(\hat{c}^*_\mathrm{fat}\big)^\frac{q}{p} |B(z,r)|^{1-\frac{q}{p}} \cp_{p-\varepsilon}\left( \overline{B\big(z,\tfrac{r}{2}\big)} ; B(z,r) \right)^\frac{q}{p}
        \right\}
        \\&\geq
        c(n,p,q) \big(\hat{c}^*_\mathrm{fat}\big)^\frac{q}{p} \cp_{p-\varepsilon} \Big(\overline{B\big(z,\tfrac{r}{2}\big)} ; B(z,r)\Big) 
        \\&\geq
        c(n,p,q) \big(\hat{c}^*_\mathrm{fat}\big)^\frac{q}{p} \cp_{\widehat{\varphi}}^{\inf} \Big(\overline{B\big(z,\tfrac{r}{2}\big)} ; B(z,r)\Big)
    \end{align*}
    for any $0<r<\hat{r}_o$. This shows that $E$ is uniformly infimally $\widehat{\varphi}$-fat.
    It remains to show that this is equivalent to uniform infimal $\varphi^{1-\varepsilon/p}$-fatness for $\delta=\varepsilon/p$.
    To this end, consider any compact subset $E$ of an open set $O \subset \R^n$, and a function $u \in \Lip(\overline{O})$ that satisfies $u \geq t\chi_E$ for some $t>0$ and $\left.u\right|_{\partial O} = 0$.
    Since \eqref{aux-phiselfimproving1} and \eqref{aux-phiselfimproving2} yield
    \begin{equation*}
        \tfrac{1}{4} \int_O \frac{\varphi^{1-\delta}(x,|\nabla u|)}{(\varphi_O^-)^{1-\delta}(t)} \,\dx
        \leq
        \int_O \frac{\widehat{\varphi}(x,|\nabla u|)}{\widehat{\varphi}_O^-(t)} \,\dx
        \quad \text{and} \quad
        \int_O \frac{\widehat{\varphi}(x,|\nabla u|)}{\widehat{\varphi}_O^-(t)} \,\dx
        \leq
        4\int_O \frac{\varphi^{1-\delta}(x,|\nabla u|)}{(\varphi_O^-)^{1-\delta}(t)} \,\dx,
    \end{equation*}
    we immediately deduce that
    \begin{equation*}
        \tfrac{1}{4}\cp_{\varphi^{1-\varepsilon/p}}^{\inf}(E;O)
        \leq
        \cp_{\widehat{\varphi}}^{\inf}(E;O)
        \leq 
        4 \cp_{\varphi^{1-\varepsilon/p}}^{\inf}(E;O).
    \end{equation*}
    Together with the penultimate inequality, this shows that $E$ is uniformly infimally $\varphi^{1-\varepsilon/p}$-fat with constant $\hat{c}_\mathrm{fat} := c(n,p,q) \big(\hat{c}^*_\mathrm{fat}\big)^\frac{q}{p}$, i.e., the claim of the lemma.
\end{proof}


\section{Auxiliary results}
\label{sec:aux}

In this section, we collect auxiliary results that will be needed in the rest of the paper.
First of all, note that by assumption \eqref{eq:gap-exponents}, the exponent $q$ always lies in one of the following intervals, depending on $p$:
\begin{equation}
    \left\{
    \begin{array}{ll}
        1< q \leq p + \frac{\alpha p}{n} = \frac{n+\alpha}{n} p \leq \frac{np}{n-p} & \text{if } 1<p<n,\\[5pt]
        1< q < +\infty & \text{if } p \geq n. 
    \end{array}
    \right.
    \label{eq:q-less-Sobolev-exponent}
\end{equation}
\\
We begin with the following Maz'ya inequality for the $p$-energy. While the statement holds for general $p$-quasicontinuous functions, see \cite[Theorem 5.47]{KLV21}, we only formulate it for Lipschitz continuous functions. For a more general version see \cite[Theorem 6.21]{BjBj}.

\begin{lemma}
\label{lem:classical-Mazya}
Assume that $1<p<+\infty$ and that $u \in \Lip(\R^n)$.
Further, let $1\leq s \leq \frac{np}{n-p}$ if $p<n$ and $1\leq s < +\infty$ if $p\ge n$.
Then, for any ball $B(z,r) \subset \R^n$ there exists a constant $c = c(n,p,s)$ such that
\begin{equation*}
    \bigg( \bint_{B(z,r)} |u|^s \, \dx \bigg)^{\frac{p}{s}} 
    \leq 
    \frac{c}{\cp_p\left(\{u=0\} \cap \overline{B\left(z,\frac{r}{2}\right)};B(z,r)\right)} \int_{B(z,r)} |\nabla u|^p \, \dx.
\end{equation*}
\qed
\end{lemma}

For the Poincaré inequality for functions $u$ with zero boundary values involving the double-phase integrand $\varphi$, we refer to \cite[Theorem 2.13]{Ok2017}, see also \cite[Theorem 1.6 \& Remark 2]{CM15_2} for the case $\frac{q}{p} < 1+\frac{\alpha}{n}$.
By Hölder's inequality, the statement of the following lemma holds in particular for $\theta=1$.

\begin{lemma}
\label{lem:poincare-zero-boundary-values}
Let $\Omega \subset \R^n$ be a bounded open set and consider a ball $B(z,r) \subset \Omega$ with $0<r \le 1$.
Assume that \eqref{as:a} holds and let $1<p<q<+\infty$ be such that \eqref{eq:gap-exponents} holds.
Then, there exists a constant $c = c(n,p,q) \ge 1$ and an exponent $0 < \theta < 1$ depending only on $n,p$ and $q$ such that
\begin{align*}
    \bint_{B(z,r)} \varphi \left( x, \frac{|u - (u)_{B(z,r)}|}{r} \right) \,\dx
    \leq
    c \left( 1 + [a]_{0;\alpha} \|\nabla u\|_{L^p(\Omega)}^{q-p} \right)
    \bigg( \bint_{B(z,r)} \varphi ( x, |\nabla u| )^{\theta} \,\dx \bigg)^{\frac{1}{\theta}}
\end{align*}
holds for any $u \in W^{1,p}(\Omega)$. Moreover, in the above estimate one can replace $u - (u)_{B(z,r)}$ by $u$ if $u \in W^{1,p}_0(B(z,r))$.
\qed
\end{lemma}

Next, for the $p$-Hardy inequality, cf.~\cite[Theorem 6.25]{KLV21}. The following result was first proven for $p = 2$ in \cite{A86}, and it was generalized for $p \ge 1$ in \cite{L88, W90}. While global $p$-fatness is assumed in \cite[Theorem 6.25]{KLV21}, the proof shows that an analogous statement holds for bounded sets $\Omega$ under the weaker local $p$-fatness condition.

\begin{lemma}
\label{lem:p-Hardy}
Assume that $\Omega \subset \R^n$ is a bounded open set such that $\R^n \setminus \Omega$ is locally $p$-fat with parameters $c_\mathrm{fat}>0$ and $r_o\in (0,1]$, with $1 < p < +\infty$.
Then, there exists a constant $c=c(n,p,r_o, c_\mathrm{fat})$ such that the $p$-Hardy inequality
$$
    \int_\Omega \frac{|u(x)|^p}{\dist(x,\partial\Omega)^p} \,\dx
    \leq
    c \int_\Omega |\nabla u|^p \,\dx
$$
holds for any $u \in \Lip_0(\Omega)$.
\qed
\end{lemma}

In the following, we use a Whitney cube decomposition of $\Omega$, i.e.~a decomposition into cubes whose diameter is comparable to their distance from $\partial\Omega$. The following statement can be found in \cite[Chapter VI]{Stein}, see also \cite[Lemma 8.9]{KLV21}.
\begin{lemma}
\label{lem:Whitney}
Assume that $\Omega \subsetneq \R^n$ is a non-empty open set.
Then there exist pairwise disjoint half-open dyadic cubes $Q_i \subset \Omega$, $i \in \N$, such that $\Omega = \bigcup_{i=1}^\infty Q_i$ and
$$
	\diam(Q_i)
	\leq
	\dist(Q_i,\partial\Omega)
	\leq
	4\diam(Q_i)
	\quad \forall \, i \in \N.
$$
\qed
\end{lemma}

\begin{remark}
\label{rem:bounded-overlap}
Let $\Omega \subset \R^n$ be an open set and assume that $\W$ is a Whitney cube decomposition of $\Omega$ according to Lemma \ref{lem:Whitney}.
For $k \in \Z$, define
\begin{equation*}
    \W_k := \big\{ Q \in \W : \diam(Q) = 2^k \big\}.
\end{equation*}
Checking the construction in \cite[Chapter VI]{Stein}, we observe that for any $k \in \Z$ the Whitney cubes $Q \in \W_k$ have vertices in $2^k\Z^n$.
Therefore, for any fixed constant $C \geq 1$ and $k \in \Z$, the enlarged cubes $CQ$ with side length $C2^k$ and the same center as $Q$ have bounded overlap depending only on $C$ and $n$ in the sense that
$$
	\sum_{Q \in \W_k} \chi_{CQ}(x)
	\leq
	c(n,C)
	\qquad \forall \, x \in \Omega.
$$
\end{remark}

The following result concerning the restricted centered maximal operator, defined in Definition \ref{def:max_op}, can be shown by combining \cite[Remark 1.29]{KLV21}, \cite[Theorem 3.4]{KLV21} and \cite[Inequality (4.1)]{KLV21}, see also \cite[Inequality (6.12)]{KLV21}.

\begin{lemma}
\label{lem:aux-equivalence}
Let $u \in \Lip(\R^n)$, $z \in \R^n$, $r>0$ and $s\in (1,\infty)$.
Then, we have that
\begin{equation*}
    |u(z)-(u)_{B(z,r)}|^s \leq C(n,s) M_{2r}(|r\nabla u|^s )(z).
\end{equation*}
\qed
\end{lemma}

Now we recall the following iteration lemma, which is a special case of \cite[Lemma 6.1]{G03}.

\begin{lemma}
\label{lem:iteration}
Let  $\vartheta \in [0,1)$, $A,B \geq 0$ and $\alpha>0$.
Further, let $0<a<b$ and assume that $\phi \colon [a,b] \to [0,+\infty)$ is a bounded function satisfying
$$
    \phi(t)
    \leq
    \vartheta \phi(s) + \frac{A}{(s-t)^\alpha} + B
    \qquad \text{for all } a \leq t < s \leq b.
$$
Then, there exists a constant $c=c(\alpha,\vartheta)$ such that
$$
    \phi(a)
    \leq
    c \left[ \frac{A}{(b-a)^\alpha} + B\right].
$$
\qed
\end{lemma}

The higher integrability will be proven by means of a Gehring-type inequality. According to \cite[Corollary 6.1 \& Theorem 6.6]{G03} and \cite[Lemma 3.7]{Karp21} we have the following result.
\begin{lemma}
\label{Gehring's Lemma}
    Let $z \in \R^n$ and let $r > 0$. Let $g\in L^1(B(z,r))$ and $h\in L^s(B(z,7r))$ for some $s>1$ be non-negative functions. If there exist $\theta\in (0,1)$ and $c_*>0$ such that
    \begin{equation*}
        \bint_{B(z,r)} g \, \dx \leq c_* \Bigg[ \bigg(\bint_{B(z,6r)} g^\theta \, \dx \bigg)^{\frac{1}{\theta}} + \bint_{B(z,6r)} h \, dx \Bigg]
    \end{equation*}
    then there exist $c=c(n,s,\theta,c_*)>0$ and $\sigma=\sigma(n,s,\theta,c_*)>1$ for which there holds
    \begin{equation*}
        \bigg(\bint_{B(z,r)} g^\sigma \, \dx \bigg)^{\frac{1}{\sigma}} \leq c \Bigg[ \bint_{B(z,8r)} g \, \dx + \bigg(\bint_{B(z,8r)} h^\sigma \, \dx \bigg)^{\frac{1}{\sigma}} \Bigg].
    \end{equation*}
    \qed
\end{lemma}


\section{A Maz'ya type inequality involving the double-phase functional}
\label{sec:Mazya-type-ineq}

In this section, we show the following Maz'ya type inequality for certain functionals with $(p,q)$-growth. 

\begin{lemma}[Maz'ya type inequality]
\label{lem:protoMazya-type-ineq}
Let $z \in \R^n$, $r > 0$, $u \in \Lip(B(z,r))$. Assume \eqref{as:a} holds and let $1<p<q<+\infty$ be such that \eqref{eq:gap-exponents} holds. 
Let $1\leq m \leq p$ and if $m<n$, additionally assume that $q\leq \frac{nm}{n-m}$.
Further, define
\begin{equation*}
    Z:=\left\{x \in \overline{B\left(0,\tfrac{1}{2}\right)} : u(z+rx)=0\right\}
\end{equation*}
and let $y\in \R^n$ and $R>0$ be arbitrary.
Then there exists a constant $c=c(n,p,q,m)$ such that
\begin{align*}
    \bint_{B(z,r)} & \varphi^\pm_{y,R}(|u|) \,\dx
    \\ &\leq
    \frac{c(n,p,q,m) \big[ 1 + a^\pm_{y,R} \cp_m(Z;B(0,1))^{\frac{p-q}{m}} (r \nabla u)_{B(z,r),m}^{q-p} \big]}{\cp_m(Z;B(0,1))^{\frac{p}{m}}} 
    \bigg( \bint_{B(z,r)} |r \nabla u|^m \,\dx \bigg)^{\frac{p}{m}}.
\end{align*}
\end{lemma}

\begin{proof}
By \eqref{eq:equivalence-density-aux}, we can rewrite the classical Maz'ya inequality of Lemma \ref{lem:classical-Mazya}, where instead of $p$ we consider the exponent $m$, as
$$
    \bint_{B(z,r)} |u|^s \, \dx 
    \leq 
    \bigg( 
    \frac{c(n,m,s) r^{m-n}}{\cp_m(Z;B(0,1))} \int_{B(z,r)} |\nabla u|^m \, \dx
    \bigg)^{\frac{s}{m}}.
$$
Since Lemma \ref{lem:classical-Mazya} is applicable with $s=p$ and with $s=q$ due to \eqref{eq:gap-exponents} and the assumption of the lemma, we obtain that
\begin{align*}
    \bint_{B(z,r)} & \varphi^\pm_{y,R}(|u|) \,\dx
    =
    \bint_{B(z,r)} \left( |u|^p + a_{y,R}^\pm |u|^q \right) \dx 
    \\&\leq
    \bigg( 
    \frac{c(n,p,m) r^{-n}}{\cp_m(Z;B(0,1))} \int_{B(z,r)} |r \nabla u|^m \,\dx
    \bigg)^{\frac{p}{m}} 
    \\&\phantom{=}
    + a_{y,R}^\pm
    \bigg( 
    \frac{c(n,q,m) r^{-n}}{\cp_m(Z;B(0,1))} \int_{B(z,r)} |r\nabla u|^m \, \dx
    \bigg)^{\frac{q}{m}}
    \\&=
    \frac{c(n,p,m)}{\cp_m(Z;B(0,1))^{\frac{p}{m}}} \bigg( \bint_{B(z,r)} |r\nabla u|^m \, \dx
    \bigg)^{\frac{p}{m}}  
    \\&\phantom{=}
    +\frac{c(n,q,m) a^\pm_{y,R} \cp_m(Z;B(0,1))^{\frac{p-q}{m}} (r \nabla u)_{B(z,r),m}^{q-p}}{\cp_m(Z;B(0,1))^{\frac{p}{m}}} 
    \bigg( \bint_{B(z,r)} |r\nabla u|^m \, \dx
    \bigg)^{\frac{p}{m}}
    \\ &\leq
    \frac{c(n,p,q,m) \big[ 1 + a^\pm_{y,R} \cp_m(Z;B(0,1))^{\frac{p-q}{m}} (r \nabla u)_{B(z,r),m}^{q-p} \big]}{\cp_m(Z;B(0,1))^{\frac{p}{m}}} 
    \bigg( \bint_{B(z,r)} |r \nabla u|^m \,\dx \bigg)^{\frac{p}{m}}.
\end{align*}
This concludes the proof of the lemma.
\end{proof}

\begin{remark}
\label{rmk:choiceofm}
    The preceding premise $q\leq \frac{nm}{n-m}$ of the lemma is somewhat opaque. Accordingly, in the following we will derive a parameter range for $m$ where this inequality is satisfied. We start with the observation that the requirement \eqref{eq:gap-exponents} on the ratio $q/p$ implies the upper bound
    \begin{equation*}
        q\leq p\left( 1+\frac{\alpha}{n} \right)\leq p+\frac{p}{n}.
    \end{equation*}
    From this upper bound we can deduce for $1<p<n$ and $0\leq \hat{p} \leq p-1$ that
    \begin{equation*}
        \frac{n(p-\hat{p})}{n-p+\hat{p}} - q 
        \geq 
        \frac{n(p-\hat{p})}{n-p+\hat{p}} - \frac{np+p}{n}
        =
        \frac{p^2n-pn+p^2-\hat{p}(n^2+pn+p)}{n(n-p+\hat{p})}
    \end{equation*}
    and the latter is non-negative if
    \begin{equation*}
        0\leq \hat{p} \leq \frac{p^2n-pn+p^2}{n^2+pn+p}. 
    \end{equation*}
    We conclude that the assumption $q\leq \frac{nm}{n-m}$ of the previous lemma is satisfied if there holds 
    \begin{equation*}
        \max\left\{ 1, \frac{pn^2 +pn}{n^2+pn+p} \right\} \leq m \leq p.
    \end{equation*}
    In particular this parameter range is nonempty and contains $p$.
\end{remark}

\begin{example}
\label{rem:counterexample-Mazya}
At this stage, a natural question is whether a Maz'ya type inequality of the form
$$
	\bint_{B(z,r)} \varphi(x,|u|) \,\dx
	\leq
	\frac{C}{\cp_\varphi\left(\{u=0\} \cap \overline{B\left(z,\frac{r}{2}\right)}, B(z,r)\right)}
	\int_{B(z,r)} \varphi(x,|\nabla u|) \,\dx
$$
could hold with a capacity $\cp_\varphi$ that is better adapted to the double-phase integrand $\varphi$ (e.g.~the variational $\varphi$-capacity defined in \cite{BHH18}) and a constant depending on the data $n$, $p$, $q$, $a$ and possibly on $u$ or $\nabla u$ in a reasonable way.
In particular, we expect that $|u|$ and $|\nabla u|$ should appear on the right-hand side with a non-negative exponent.
However, the following counterexample shows that such an inequality can only hold with a notion of capacity satisfying $\cp_\varphi\left(\{u=0\} \cap \overline{B\left(z,\frac{r}{2}\right)}, B(z,r)\right) \leq \cp_p\left(\{u=0\} \cap \overline{B\left(z,\frac{r}{2}\right)}, B(z,r)\right)$ up to a multiplicative constant.
The problem already appears for the $x$-independent integrand
$$
    \varphi_o(s) := s^p + a_o s^q
$$
with $p \leq n$, $a_o>0$, and Lipschitz functions $u$ with small function values and small gradient in small balls.
\\
Indeed, consider balls $B(0,r)$ with $0<r \leq 1$ and functions $u_r \in \Lip(\R^n)$ given by $u_r(x) := \big(|x|-\frac{r}{2}\big)_+$.
Since $0 \leq u_r \leq 1$ in $B(0,r)$, we have that $u_r^q \leq u_r^p$ in $B(0,r)$ and thus that
$$
	\bint_{B(0,r)} u_r^p \,\dx
	\leq
	\bint_{B(0,r)} \varphi_o(u_r) \,\dx
	\leq
	(1+a_o) \bint_{B(0,r)} u_r^p \,\dx.
$$
Using polar coordinates, we compute that
$$
	\bint_{B(0,r)} u_r^p \,\dx
	=
	\bint_{B(0,r)} \Big(|x|-\frac{r}{2}\Big)_+^p \,\dx
	=
	c(n) r^{-n} \int_\frac{r}{2}^r \Big(t-\frac{r}{2}\Big)^p t^{n-1} \,\dt.
$$
Since
\begin{align*}
	\frac{1}{2^{n+p}(p+1)} r^{n+p}
	&=
	\Big(\frac{r}{2}\Big)^{n-1} \cdot \frac{1}{p+1} \Big(\frac{r}{2}\Big)^{p+1} \\
	&=
	\Big(\frac{r}{2}\Big)^{n-1} \int_\frac{r}{2}^r \Big(t-\frac{r}{2}\Big)^p \,\dt \\
	&\leq
	\int_\frac{r}{2}^r \Big(t-\frac{r}{2}\Big)^p t^{n-1} \,\dt \\
	&\leq
	r^{n-1} \int_\frac{r}{2}^r \big( t-\frac{r}{2}\Big)^p \,\dt
	=
	\frac{1}{2^{p+1}(p+1)} r^{n+p},
\end{align*}
combining the preceding three inequalities yields
$$
	c(n,p) r^p
        \le
        \bint_{B(0,r)} \varphi_o(u_r) \,\dx
	\le
	c(n,p,a_o) r^p.
$$
Next, note that $\nabla u_r(x) = \frac{x}{|x|} \chi_{\R^n \setminus B \left(0,\frac{r}{2}\right)}(x)$.
This means that
$$
	\int_{B(0,r)} \varphi_o(|\nabla u_r|) \,\dx
	=
	\int_{B(0,r) \setminus B\left(0,\frac{r}{2}\right)} (1 + a_o) \,\dx
	=
	c(n) (1+a_o) \Big(r^n - \Big(\frac{r}{2}\Big)^n \Big)
    =
    c(n) (1+a_o) r^n.
$$
We conclude that an inequality of the form
$$
    \bint_{B(0,r)} \varphi_o(u_r) \,\dx
    \leq
    Q(r) \int_{B(0,r)} \varphi_o(|\nabla u_r|) \,\dx
$$
can only hold for a quantity
$$
    Q(r)
    \geq
    c(n,p,a_o) r^{p-n}
    \ge
    \frac{c(n,p,a_o)}{\cp_p\left(\{u_r=0\} \cap \overline{B\left(0,\frac{r}{2}\right)}, B(0,r)\right)}.
$$
For the last estimate, we have taken \eqref{eq:scaling-p-capacity} and the fact that $\{u_r=0\} = \overline{B\left( 0,\frac{r}{2} \right)}$ into account.

In contrast, note that the variational $\varphi$-capacity of a compact set $K$ in an open set $O \supset K$ set is defined by
$$
	\cp_\varphi(K,O)
	=
	\inf\left\{
	\int_O \varphi(|\nabla u|) \,\dx :
	u \in W^{1,\varphi}(O;\R_{\geq 0}) \cap C^0_c(O), u \geq 1 \text{ in } K
	\right\},
$$
see \cite[Proposition 20]{BHH18}.
In our case, we have that $W^{1,\varphi}(B(0,r)) = W^{1,q}(B(0,r))$.
In order to estimate $\cp_\varphi\left(\{u_r = 0\} \cap \overline{B\left(0,\frac{r}{2}\right)}, B(0,r)\right)$, first note that by Poincaré's inequality
\begin{align*}
	c(n) r^n
	=
	\left| B\left(0,\tfrac{r}{2}\right) \right|
	\leq
	\int_{B\left(0,\frac{r}{2}\right)} |u|^q \,\dx
	\leq
	c(n,q) r^q \int_{B(0,r)} |\nabla u|^q \,\dx
\end{align*}
for any $u \in W^{1,q}_0(B(0,r))$ such that $u \geq 1$ in $B\left(0,\frac{r}{2}\right)$ and, analogously,
$$
	c(n) r^n
	\leq
	c(n,p) r^p \int_{B(0,r)} |\nabla u|^p \,\dx
$$
for the same class of functions.
Recalling that $r \leq 1$ and taking the infimum over all admissible functions in the definition of $\cp_\varphi$, this shows that
\begin{equation*}
\begin{aligned}
    r^{n-q}
	=
	\max\{r^{n-p},r^{n-q}\}
    &\leq
    \frac{1}{\min\{1,a_o\}} \big( r^{n-p} + a_o r^{n-q} \big) \\
    &\leq
	c(n,p,q,a_o) \cp_\varphi\left(\{u_r=0\} \cap \overline{B\left(0,\tfrac{r}{2}\right)}, B(0,r)\right).
\end{aligned}
\end{equation*}
Since for any constant $C$ there exists a radius $r_C$ such that $C r^{q-n} < r^{p-n}$ for any $0<r<r_C$, this in turn implies that 
$\cp_\varphi\left(\{u_r=0\} \cap \overline{B\left(0,\frac{r}{2}\right)}, B(0,r)\right)^{-1}$ is essentially smaller than $r^{p-n}$.
Therefore, a Maz'ya type inequality does not hold with this notion of capacity in the denominator on the right-hand side.
\end{example}


\section{From the fatness condition to an integral Hardy inequality}
\label{sec:int-hardy}

This section is devoted to the proof of Theorem \ref{thm:int_har_ineq}, in particular proving an integral Hardy inequality assuming the fatness condition on the complement of the domain $\Omega$.
\\[0.2cm]
\begin{proof}[Proof of Theorem \ref{thm:int_har_ineq}]
Let us assume that $\diam(\Omega) \leq \frac{1}{2}$.
In that case, our goal is to prove that there exists a constant $c = c(n,p,q,r_o,c_{\mathrm{fat}})$ such that
$$
	\int_\Omega \varphi\left(x, \dist(x,\partial\Omega)^{-1}|u(x)| \right) \,\dx
	\leq
	c \left(1+[a]_{0;\alpha}\|\nabla u\|_{L^p(\Omega)}^{q-p}\right)^\frac{p}{p-1} \int_{\Omega} \varphi \left(x, \left|\nabla u(x)\right| \right) \,\dx
$$
holds for any $u \in \Lip_0(\Omega)$.\\
Provided that this is true, we can then show the claim for general domains by transformation.
Indeed, set $\lambda := \diam(\Omega)$ and $\widetilde{\Omega} := \Omega/(2\lambda)$, $\tilde{u}(y) = u(2\lambda y)/(2\lambda )$ and $\widetilde{\varphi}(y,t)=\varphi(2\lambda y,t)$.
Note that $\diam(\widetilde{\Omega}) \leq 1/2$ and that $\widetilde{\varphi}(y,t) = t^p + \tilde{a}(y)t^q$, where $\tilde{a}(y) := a(2\lambda y)$ is Hölder continuous to the exponent $\alpha$ with $[\tilde{a}]_{0;\alpha} = (2\lambda )^\alpha [a]_{0;\alpha}$.
Taking further into account that $\dist(2\lambda y,\partial\Omega) = 2\lambda  \dist(y,\partial\widetilde{\Omega})$ for any $y \in \widetilde{\Omega}$ and that $\tilde{u} \in \Lip_0(\widetilde{\Omega})$ with $\nabla \tilde{u}(y) = (\nabla u)(2\lambda y)$, the preceding inequality yields
\begin{align*}
    &\int_\Omega \varphi\left(x, \frac{|u(x)|}{\dist(x,\partial\Omega)} \right) \,\dx
    =
    \int_{\widetilde{\Omega}} \varphi\left(2\lambda y, \frac{|u(2\lambda y)|}{\dist(2\lambda y,\partial\Omega)} \right) (2\lambda )^n \,\dy
    \\&=
    (2\lambda )^n \int_{\widetilde{\Omega}} \widetilde{\varphi} \left( y, \frac{|\tilde{u}(y)|}{\dist(y,\partial\widetilde{\Omega})} \right) \,\dy
    \\&\leq
    c (2\lambda )^n \left(1+[\tilde{a}]_{0;\alpha}\|\nabla \tilde{u}\|_{L^p(\widetilde{\Omega})}^{q-p}\right)^\frac{p}{p-1} \int_{\widetilde{\Omega}} \widetilde{\varphi} \left(y, \left|\nabla \tilde{u}(y)\right| \right) \,\dy
    \\&\leq
    c \left( 1 + \lambda ^{\alpha + \frac{n(p-q)}{p}} [a]_{0;\alpha} \left( \int_{\widetilde{\Omega}} |(\nabla u)(2 \lambda y)|^p (2\lambda )^n \,\dy \right)^\frac{q-p}{p} \right)^\frac{p}{p-1}
    \\&\phantom{=}\cdot
    \int_{\widetilde{\Omega}} \varphi \left( 2\lambda y, |(\nabla u)(2\lambda y)| \right) (2\lambda )^n \,\dy
    \\&=
    c \left( 1 + \lambda ^{\alpha + \frac{n(p-q)}{p}} [a]_{0;\alpha} \|\nabla u\|_{L^p(\Omega)}^{q-p} \right)^\frac{p}{p-1} \int_\Omega \varphi(x,|\nabla u|) \,\dx,
\end{align*}
where $c = c(n,p,q,r_o,c_{\mathrm{fat}})$. \\
In order to deduce the penultimate inequality for domains with $\diam(\Omega) \leq \frac{1}{2}$, we argue in the following way.
Let $\mathcal{W}$ be a Whitney cube decomposition of $\Omega$ in the sense of Lemma \ref{lem:Whitney}.
For each cube $Q \in \mathcal{W}$, let $x_Q \in \partial\Omega$ be a point with $\dist(Q,\partial\Omega) = \dist(Q,x_Q)$ and note that the ball $B_Q := B(x_Q,r_Q)$ with $r_Q := 5 \diam(Q)$ contains $Q$.
Moreover, there exists a constant $C=C(n)$ such that the cube $CQ$, i.e.~the cube with the same center as $Q$ and $\diam(CQ) = C \diam(Q)$, contains $B_Q$.
Next, for $k \in \N$ we set
$$
	\Omega_k
	:=
	\bigcup_{Q \in \mathcal{W}_k} Q,
	\quad \text{where }
	\mathcal{W}_k
	:=
	\big\{Q \in \mathcal{W} : \diam(Q) = 2^{-k} \big\},
$$
and
$$
	\widetilde{\Omega}_k := \bigcup_{m=k}^{+\infty} \Omega_m.
$$

For $k_o := \lceil 2 + \log(C) \rceil$ and $Q \subset \Omega_k$, $k \geq k_o$, we have that $C Q \cap \Omega \subset \widetilde{\Omega}_{k-k_o}$.
Moreover, there exists $k_1 = k_1(n,r_o) > k_o$ such that $r_Q < r_o$ for any $\W \ni Q \subset \widetilde{\Omega}_{k_1}$.
Let $u \in \Lip_0(\Omega)$, $0<\beta<1$ and abbreviate $d(x) := \dist(x,\partial\Omega)$.
First, note that by the fatness condition \eqref{eq:inf-capacity-density} in conjunction with the upper bound from Lemma \ref{lem:upper-estimate-infcap}, we have that
\begin{align*}
	&\cp_p \Big( \Big\{ x \in \overline{B\big(0,\tfrac{1}{2}\big)} : u(x_Q+r_Qx)=0 \Big\}; B(0,1) \Big)
    \\ &\geq
    \cp_p \left(\left\{x \in \overline{B\left(0,\tfrac{1}{2}\right)} : x_Q + r_Q x \in \R^n \setminus \Omega \right\}; B(0,1)\right)
    \\ &\geq
	\cp_{\widetilde{\varphi}}^{\inf}\left(\left\{x \in \overline{B\left(0,\tfrac{1}{2}\right)} : x_Q + r_Q x \in \R^n \setminus \Omega \right\}; B(0,1)\right)
	\geq
	c_\mathrm{fat}.
\end{align*}
Thus, applying Lemma \ref{lem:protoMazya-type-ineq} to $u/r_Q^{1+\beta}$ with the parameter choices $(r,R,z,y,m)\equiv(r_Q,r_Q,x_Q,x_Q,p)$, for any $\W \ni Q \subset \widetilde{\Omega}_{k_1}$ we obtain that
\begin{align*}
    \int_Q &\varphi \left(x,\frac{|u(x)|}{d^{1+\beta}(x)} \right) \,\dx
    \leq
    \int_Q \varphi^+_Q \left( \frac{|u(x)|}{d^{1+\beta}(x)} \right) \,\dx \\
    &\leq
    c(q) \int_{B_Q} \varphi^+_{x_Q,r_Q} \left( r_Q^{-(1+\beta)}|u(x)| \right) \,\dx
    \nonumber \\&\leq
    c \, \frac{ 1 + a^+_{x_Q,r_Q} \cp_p \left( \left\{ x \in \overline{B\left(0,\frac{1}{2}\right)} : u(x_Q+r_Qx)=0 \right\}; B(0,1) \right)^{1-\frac{q}{p}} \left( r_Q^{-\beta} \nabla u \right)_{B_Q,p}^{q-p}}{\cp_p \left( \left\{ x \in \overline{B\left(0,\frac{1}{2}\right)} : u(x_Q+r_Qx)=0 \right\}; B(0,1) \right)}
    \nonumber\\ & \qquad \cdot 
    \int_{B_Q} \left|r_Q^{-\beta} \nabla u(x)\right|^p \,\dx
    \nonumber\\ &\leq
    c \,\left[ 1 + a^+_{x_Q,r_Q} \left( r_Q^{-\beta} \nabla u \right)_{B_Q,p}^{q-p} \right]
    \int_{B_Q} \left|r_Q^{-\beta} \nabla u(x)\right|^p \,\dx
\end{align*}
with a constant $c=c(n,p,q,c_\mathrm{fat})$.
By means of Hölder's inequality, Hölder continuity of $a$, \eqref{eq:gap-exponents} and the facts that $5 \geq r_Q \geq d(x)$ for any $x \in B_Q$, $B_Q \subset CQ$ and $u \in \Lip_0(\Omega)$, we estimate the right-hand side of the preceding inequality by
\begin{align*}
    a^+_{x_Q,r_Q} & \left( r_Q^{-\beta} \nabla u \right)_{B_Q,p}^{q-p} 
    \int_{B_Q} \left|r_Q^{-\beta} \nabla u(x)\right|^p \,\dx
    \nonumber\\ & \leq
    \left( a^-_{x_Q,r_Q} + [a]_{0,\alpha} (2r_Q)^\alpha \right)\left( r_Q^{-\beta} \nabla u \right)_{B_Q,p}^{q-p}
    \int_{B_Q} \left|r_Q^{-\beta} \nabla u(x)\right|^p \,\dx
    \nonumber\\ & \leq
    \int_{B_Q} a^-_{x_Q,r_Q} \left|r_Q^{-\beta} \nabla u(x)\right|^q \,\dx 
    + 
    c [a]_{0,\alpha} r_Q^{\alpha-\frac{n(q-p)}{p}} \left\| r_Q^{-\beta} \nabla u(x) \right\|_{L^p(B_Q)}^{q-p} 
    \int_{B_Q} \left|r_Q^{-\beta} \nabla u(x)\right|^p \,\dx
    \nonumber\\ & \leq
    c \left( 1 + [a]_{0;\alpha} \left\| d^{-\beta}(x) \nabla u(x) \right\|_{L^p(\Omega)}^{q-p} \right)
    \int_{CQ \cap \Omega} \varphi \left(x, \left|r_Q^{-\beta} \nabla u(x)\right| \right) \,\dx,
\end{align*}
with a constant $c=c(n,p,q)$. This leads us to
\begin{align}\label{eq:Hardy-aux1}
    \int_Q &\varphi \left(x,\frac{|u(x)|}{d^{1+\beta}(x)} \right) \,\dx
    \leq
    c \left( 1 + [a]_{0;\alpha} \left\| d^{-\beta}(x) \nabla u(x) \right\|_{L^p(\Omega)}^{q-p} \right)
    \int_{CQ \cap \Omega} \varphi \left(x, \left|r_Q^{-\beta} \nabla u(x)\right| \right) \,\dx
\end{align}
for any cube $\W \ni Q \subset \widetilde{\Omega}_{k_1}$ with $c=c(n,p,q,c_\mathrm{fat})$.
Summing the inequalities \eqref{eq:Hardy-aux1} related to $Q \in \W_k$ for some $k \geq k_1$, 
using that the cubes $CQ$ with $Q \in \W_k$ have uniformly bounded overlap depending only on the space dimension $n$ (see Remark~\ref{rem:bounded-overlap}), that $CQ \cap \Omega \subset \widetilde{\Omega}_{k-k_o}$ for any $Q \in \W_k$,
and finally that $r_Q = 5\sqrt{n} 2^{-k}$ for any $Q \in \W_k$, we infer
\begin{align*}
	\int_{\Omega_k} &\varphi \left(x, \frac{|u(x)|}{d^{1+\beta}(x)} \right) \,\dx\\
	&\leq
	c \left( 1 + [a]_{0;\alpha} \left\| d^{-\beta}(x) \nabla u(x) \right\|_{L^p(\Omega)}^{q-p} \right)
	\sum_{Q \in \W_k}
	\int_{CQ \cap \Omega} \varphi \left(x, \left|r_Q^{-\beta} \nabla u(x)\right| \right) \,\dx \\
	&\leq
	c \left( 1 + [a]_{0;\alpha} \left\| d^{-\beta}(x) \nabla u(x) \right\|_{L^p(\Omega)}^{q-p} \right)
	\sum_{m=k-k_o}^\infty
	\int_{\Omega_m} \varphi \left(x, \left|2^{k\beta} \nabla u(x)\right| \right) \,\dx.
\end{align*}
Now, summing over $k \geq k_1$, and changing the order of summation, we conclude that
\begin{align}
	\int_{\widetilde{\Omega}_{k_1}} &\varphi \left(x, \frac{|u(x)|}{d^{1+\beta}(x)} \right) \,\dx
	\nonumber\\
	&\leq
	c \left( 1 + [a]_{0;\alpha} \left\| d^{-\beta}(x) \nabla u(x) \right\|_{L^p(\Omega)}^{q-p} \right)
	\sum_{k=k_1}^\infty \sum_{m=k-k_o}^\infty
	\int_{\Omega_m} \varphi \left(x, \left|2^{k\beta} \nabla u(x)\right| \right) \,\dx
	\nonumber\\
	&\leq
	c \left( 1 + [a]_{0;\alpha} \left\| d^{-\beta}(x) \nabla u(x) \right\|_{L^p(\Omega)}^{q-p} \right)
	\sum_{m=k_1-k_o}^\infty 
	\int_{\Omega_m} \sum_{k=k_1}^{m+k_o} \varphi \left(x, \left|2^{k\beta} \nabla u(x)\right| \right) \,\dx.
	\label{eq:Hardy-aux2}
\end{align}
At this stage, rewriting the sum, recalling the definition of $\varphi$, and using the formula for geometric series and the fact that $0<\beta<1$, we compute
\begin{align}
	\sum_{k=k_1}^{m+k_o} \varphi \left(x, \left|2^{k\beta} \nabla u(x)\right| \right)
	&=
	\sum_{j=k_1-m}^{-1} \varphi \left(x, \left|2^{j\beta} 2^{m\beta} \nabla u(x)\right| \right)
	+ \sum_{j=0}^{k_o} \varphi \left(x, \left|2^{j\beta} 2^{m\beta} \nabla u(x)\right| \right) \nonumber\\
	&\leq
	\sum_{j=-\infty}^{0} 2^{jp\beta} \varphi \left(x, \left| 2^{m\beta} \nabla u(x)\right| \right)
	+ \sum_{j=0}^{k_o} 2^{jq\beta} \varphi \left(x, \left| 2^{m\beta} \nabla u(x)\right| \right)\nonumber\\
	&=
	\left( \frac{1}{1-2^{-p\beta}} + \frac{2^{(k_o+1)q\beta}-1}{2^{q\beta}-1} \right)
	\varphi \left(x, \left| 2^{m\beta} \nabla u(x)\right| \right) \nonumber\\
	&\leq
	\left( \frac{1}{1-2^{-p\beta}} + \frac{2^{k_o q}}{1-2^{-q\beta}} \right)
	\varphi \left(x, \left| 2^{m\beta} \nabla u(x)\right| \right)
    \label{eq:Hardy-aux2.5}
\end{align}
By the mean value theorem, for $s \in \{p,q\}$ there exists $0\leq\tilde{\beta}\leq \beta <1$ such that
$$
	1-2^{-s\beta}
	=
	2^{-s 0}-2^{-s\beta}
	=
	-\beta \left(\frac{\d}{\d\beta} 2^{-s\beta}\right)(\tilde{\beta})
	=
	\beta \ln(2) s 2^{-s\tilde{\beta}}
	\ge
	\beta \ln(2) s 2^{-s},
$$
which is equivalent to
$$
	\frac{1}{1-2^{-s\beta}}
	\leq
	\frac{2^s}{\ln(2)s \beta}.
$$
Inserting this estimate into the right-hand side of \eqref{eq:Hardy-aux2.5} and recalling that $k_o$ depends only on $n$, we find that
$$
	\sum_{k=k_1}^{m+k_o} \varphi \left(x, \left|2^{k\beta} \nabla u(x)\right| \right)
	\leq
	c(n,p,q) \frac{1}{\beta} \varphi \left(x, \left| 2^{m\beta} \nabla u(x)\right| \right).
$$
Plugging the preceding inequality in turn into \eqref{eq:Hardy-aux2} and using that
$$
	2^{-m} \geq \frac{1}{2}\diam(Q) \geq \frac{1}{4} \diam(Q) + \frac{1}{16} \dist(Q,\partial\Omega)
	\geq
	\frac{1}{16} d(x)
$$
for any $x \in Q \subset \Omega_m$, we conclude that
\begin{align}
	\int_{\widetilde{\Omega}_{k_1}} &\varphi \left(x, \frac{|u(x)|}{d^{1+\beta}(x)} \right) \,\dx
	\nonumber \\
	&\leq
	c \left( 1 + [a]_{0;\alpha} \left\| d^{-\beta}(x) \nabla u(x) \right\|_{L^p(\Omega)}^{q-p} \right) \beta^{-1}
	\sum_{m=k_1-k_o}^\infty 
	\int_{\Omega_m} \varphi \left(x, \left| 2^{m\beta} \nabla u(x)\right| \right) \,\dx
	\nonumber\\
	&\leq
	c \left( 1 + [a]_{0;\alpha} \left\| d^{-\beta}(x) \nabla u(x) \right\|_{L^p(\Omega)}^{q-p} \right) \beta^{-1}
	\sum_{m=k_1-k_o}^\infty 
	\int_{\Omega_m} \varphi \left(x, \left| d^{-\beta}(x) \nabla u(x)\right| \right) \,\dx
	\nonumber\\
	&\leq
	c \left( 1 + [a]_{0;\alpha} \left\| d^{-\beta}(x) \nabla u(x) \right\|_{L^p(\Omega)}^{q-p} \right) \beta^{-1} 
	\int_{\Omega} \varphi \left(x, \left| \frac{\nabla u(x)}{d^\beta(x)}\right| \right) \,\dx.
	\label{eq:Hardy-aux3}
\end{align}
Moreover, since $d(x) \geq 2^{-k_1}$ in the complement of $\widetilde{\Omega}_{k_1}$, $k_1$ depends only on $n$ and $r_o$, and we have that $\beta <1$, by applying the Poincaré inequality for functions with zero boundary values of Lemma \ref{lem:poincare-zero-boundary-values} (since $\spt(u) \Subset \Omega$ and $\diam(\Omega) \leq \frac{1}{2}$, we may replace $\Omega$ with by a suitable ball $B(x_o,1)$ as the domain of integration) we obtain that
\begin{align}
	\int_{\Omega \setminus \widetilde{\Omega}_{k_1}} \varphi \left(x, \frac{|u(x)|}{d^{1+\beta}(x)} \right) \,\dx
	&\leq
	\int_{\Omega \setminus \widetilde{\Omega}_{k_1}} \varphi \left(x, 2^{2k_1}|u(x)| \right) \,\dx
	\nonumber \\
	&\leq
	c \left( 1 + [a]_{0;\alpha} \left\| \nabla u(x) \right\|_{L^p(\Omega)}^{q-p} \right) 
	\int_{\Omega} \varphi \left(x, \left| \nabla u(x)\right| \right) \,\dx
	\label{eq:Hardy-aux4}
\end{align}
with $c=c(n,p,q,r_o)$.
Adding \eqref{eq:Hardy-aux3} and \eqref{eq:Hardy-aux4} and using that $d(x) \leq 1$ and $\beta < 1$, we have that
\begin{align*}
	\int_{\Omega} &\varphi \left(x, \frac{|u(x)|}{d^{1+\beta}(x)} \right) \,\dx
	\\
	&\leq
	c \left( 1 + [a]_{0;\alpha} \left\| d^{-\beta}(x) \nabla u(x) \right\|_{L^p(\Omega)}^{q-p} \right) \beta^{-1} 
	\int_{\Omega} \varphi \left(x, \left| \frac{\nabla u(x)}{d^\beta(x)}\right| \right) \,\dx
\end{align*}
with a constant $c=c(n,p,q,r_o,c_\mathrm{fat})$.
Next, replacing $u$ in the previous inequality by $v := d^\beta u \in \Lip_0(\Omega)$ and using that
$$
	|\nabla v|
	\leq
	\beta d^{\beta-1} u + d^\beta |\nabla u|
$$
leads us to
\begin{align}
	\int_{\Omega} \varphi \left(x, \frac{|u(x)|}{d(x)} \right) \,\dx
	&=
	\int_{\Omega} \varphi \left(x, \frac{|v(x)|}{d^{1+\beta}(x)} \right) \,\dx
	\nonumber\\
	&\leq
	c \left( 1 + [a]_{0;\alpha} \left\| d^{-\beta}(x) \nabla v(x) \right\|_{L^p(\Omega)}^{q-p} \right) \beta^{-1} 
	\int_{\Omega} \varphi \left(x, \left| \frac{\nabla v(x)}{d^\beta(x)}\right| \right) \,\dx
	\nonumber\\
	&\leq
	c \left( 1 + [a]_{0;\alpha} \left\| \beta d^{-1}(x) u(x) \right\|_{L^p(\Omega)}^{q-p}  + [a]_{0;\alpha} \|\nabla u\|_{L^p(\Omega)}^{q-p}\right) \beta^{-1}
	\nonumber\\
	&\phantom{=}
	\cdot  
	\left(	
	\int_{\Omega} \varphi \left(x, \left| \beta \frac{u(x)}{d(x)}\right| \right) \,\dx
	+ \int_{\Omega} \varphi \left(x, \left|\nabla u(x)\right| \right) \,\dx \right)
	\nonumber\\
	&\leq
	c \left( 1 + [a]_{0;\alpha} \|\nabla u\|_{L^p(\Omega)}^{q-p}\right)
	\label{eq:Hardy-aux6}\\
	&\phantom{=}
	\cdot
	\left(	
	\beta^{p-1}\int_{\Omega} \varphi \left(x, \left| \frac{u(x)}{d(x)}\right| \right) \,\dx
	+ \beta^{-1} \int_{\Omega} \varphi \left(x, \left|\nabla u(x)\right| \right) \,\dx \right).
	\nonumber
\end{align}
In the penultimate line we have used that, by Theorem \ref{thm:equivalence-p-fatness} and thanks to the assumption that \eqref{eq:inf-capacity-density} holds, $\R^n \setminus \Omega$ is $p$-fat and thus the $p$-Hardy inequality in Lemma \ref{lem:p-Hardy} holds.
At this stage, we derive a Hardy inequality by reabsorbing the first integral on the right-hand side of \eqref{eq:Hardy-aux6} into the left-hand side.
Indeed, we choose
$$
	\beta := \left[ 2 c \left(1+[a]_{0;\alpha}\|\nabla u\|_{L^p(\Omega)}^{q-p}\right) \right]^{-\frac{1}{p-1}} < 1,
$$
where we assumed without loss of generality that $c > \frac{1}{2(1+[a]_{0;\alpha}\|\nabla u\|_{L^p(\Omega)}^{q-p})}$. Indeed, in the opposite case we can estimate $c$ from above by the same number; then, choosing $\beta := 2^{-\frac{1}{p-1}}$ we have the desired statement.
\\
This choice of $\beta$ yields
$$
	\int_{\Omega} \varphi \left(x, \frac{|u(x)|}{d(x)} \right) \,\dx
	\leq
	\frac12\int_{\Omega} \varphi \left(x, \frac{|u(x)|}{d(x)} \right) \,\dx
	+ c \left( 1 + [a]_{0;\alpha} \|\nabla u\|_{L^p(\Omega)}^{q-p}\right) \beta^{-1} \int_{\Omega} \varphi \left(x, \left|\nabla u(x)\right| \right) \,\dx
$$
and hence
\begin{equation}
\label{eq:Hardy-aux7}
	\int_{\Omega} \varphi \left(x, \frac{|u(x)|}{d(x)} \right) \,\dx
	\leq
	c \left(1+[a]_{0;\alpha}\|\nabla u\|_{L^p(\Omega)}^{q-p}\right)^\frac{p}{p-1} \int_{\Omega} \varphi \left(x, \left|\nabla u(x)\right| \right) \,\dx,
\end{equation}
holds with a constant $c=c(n,p,q,r_o,c_\mathrm{fat})$ for any $u \in \Lip_0(\Omega)$.
This concludes the proof of the theorem.
\end{proof}


\section{Equivalence of fatness, the boundary Poincaré inequality, and pointwise Hardy inequalities}
\label{proof_thm_eq}

In this section we finally show Theorem \ref{thm-equivalence}, in particular proving that the infimal $\varphi$-fatness condition \eqref{eq:inf-capacity-density}, the boundary Poincaré inequality and the pointwise Hardy inequality are all equivalent. Moreover, in Remark \ref{rmk:phases_equiv} we give more details on what happens considering being in the two different phases of the functional.
\\[0.2cm]
\begin{proof}[Proof of Theorem \ref{thm-equivalence}]
\\
As a first thing, note that the equivalence of \eqref{thm:equivalence-0} and \eqref{thm:equivalence-1} has already been proven in Theorem \ref{thm:equivalence-p-fatness}.
\\
Now, we show the implication from \eqref{thm:equivalence-1} to \eqref{thm:equivalence-2}.
To this end, fix $z\in \R^n \setminus \Omega$, $y\in \R^n$, $0<r < r_o$, $R>0$, and $u\in \Lip_0(\Omega)$.
For the extension of $u$ by zero, we have that $u \in \Lip(B(z,r))$.
For simplicity, denote
$$
    Z:=
    \left\{x \in \overline{B\left(0,\tfrac{1}{2}\right)} : u(z+rx)=0\right\}
$$
By Lemma \ref{lem:upper-estimate-infcap} and \eqref{thm:equivalence-1} we have that
\begin{align*}
	\cp_p(Z; B(0,1))
    \geq
    \cp_{\widetilde{\varphi}}^{\inf}(Z; B(0,1))
    \geq
	\cp_{\widetilde{\varphi}}^{\inf}\left(\left\{x \in \overline{B\left(0,\tfrac{1}{2}\right)} : z+rx \in \R^n \setminus \Omega \right\}; B(0,1)\right)
	\geq
	c_\mathrm{fat}.
\end{align*}
Thus, applying Lemma \ref{lem:protoMazya-type-ineq} with the particular choice $m=p$ and estimating the factor on the right-hand side by the preceding inequality, we obtain that
\begin{align*}
    \int_{B(z,r)} \varphi^\pm_{y,R}(|u|) \,\dx
    & \leq
    \frac{c \left[ 1 + a^\pm_{y,R} \cp_p(Z;B(0,1))^{\frac{p-q}{p}} (r\nabla u)_{B(z,r),p}^{q-p} \right]}{\cp_p(Z;B(0,1))} 
    \int_{B(z,r)} |r \nabla u|^p \,\dx \\
    &\leq
    \frac{c \left[ 1 + a^\pm_{y,R} c_\mathrm{fat}^{\frac{p-q}{p}} (r\nabla u)_{B(z,r),p}^{q-p} \right]}{c_\mathrm{fat}} 
    \int_{B(z,r)} |r \nabla u|^p \,\dx.
\end{align*}
Dividing by $|B(z,r)|$ yields \eqref{thm:equivalence-2}.
        
We proceed by establishing that \eqref{thm:equivalence-2} implies \eqref{thm:equivalence-3}.
To this end, let $z\in \Omega$ with $r_z:=\dist(z,\partial\Omega)<r_o$, $y\in \R^n$, $R>0$ and $u\in \Lip_0(\Omega)$.
Without loss of generality, we assume that $u \geq 0$ (otherwise one can replace $u$ by $|u|$).
We choose $w\in \partial \Omega$ such that $|w-z|=r_z$.
By the fact that $B(w,r_z) \subset B(z,2r_z)$, Jensen's inequality, \eqref{thm:equivalence-2}, \eqref{eq:q-less-Sobolev-exponent} and the classical $(p,p)$- and $(q,p)$-Poincaré inequalities (see for instance \cite[Theorem 3.14]{KLV21}), we obtain that
\begin{align*}
    \varphi^\pm_{y,R} &\big((u)_{B(z,r_z)} \big) \\
    &\leq
    c(q) \big[ \varphi^\pm_{y,R}\big((u)_{B(z,r_z)} - (u)_{B(z,2r_z)} \big) 
    +
    \varphi^\pm_{y,R}\big((u)_{B(z,2r_z)} - (u)_{B(w,r_z)} \big)
    +
    \varphi^\pm_{y,R}\big((u)_{B(w,r_z)} \big) \big]
    \\ & \leq
    c(n,q) \bigg[ \bint_{B(z,2r_z)} |u(x) - (u)_{B(z,2r_z)}|^p \, \dx
    +  a^\pm_{y,R} \bint_{B(z,2r_z)} |u(x) - (u)_{B(z,2r_z)}|^q \, \dx
    \\&\phantom{=}\qquad
    + \bint_{B(w,r_z)} \varphi^\pm_{y,R} (u) \, \dx \bigg]
    \\ & \leq
    c(n,p,q,c_\mathrm{fat}) \left[ 1 + a^\pm_{y,R} (r_z\nabla u)_{B(z,2r_z),p}^{q-p} + a^\pm_{y,R} (r_z\nabla u)_{B(w,r_z),p}^{q-p}\right]
    \bint_{B(z,2r_z)} |r_z\nabla u|^p \, \dx
    \\&\leq
    c(n,p,q,c_\mathrm{fat}) \left[ 1 + a^\pm_{y,R} (r_z\nabla u)_{B(z,2r_z),p}^{q-p}\right]
    \bint_{B(z,2r_z)} |r_z\nabla u|^p \, \dx,
\end{align*}
i.e.~the claim of \eqref{thm:equivalence-3}.

Next, we show that \eqref{thm:equivalence-3} yields \eqref{thm:equivalence-4}.
Here, we use the same notation as in the previous step.
Using Lemma~\ref{lem:aux-equivalence}, \eqref{thm:equivalence-3}, and the fact that $(r_z \nabla u)_{B(z,2r_z),p}^{q-p} \leq M_{2r_z}(|r_z \nabla u|^p)^{\frac{q-p}{p}}(z)$, we conclude that
\begin{align*}
    \varphi^\pm_{y,R}(u(z)) 
    & \leq 
    c(q) \big[ \varphi^\pm_{y,R}\big(u(z)- (u)_{B(z,r_z)} \big) + \varphi^\pm_{y,R}\big((u)_{B(z,r_z)} \big) \big]
    \\& = 
    c(q) \big[ |u(z)- (u)_{B(z,r_z)}|^p + a^\pm_{y,R} |u(z)- (u)_{B(z,r_z)}|^q + \varphi^\pm_{y,R}\big((u)_{B(z,r_z)} \big) \big]
    \\ & \leq 
    c(n,p,q,c_\mathrm{fat}) \bigg[ M_{2r_z}(|r_z\nabla u|^p )(z) + a^\pm_{y,R} M_{2r_z}( |r_z\nabla u|^p )^{\frac{q}{p}}(z)
    \\ & \qquad\qquad\quad +
    \left( 1 + a^\pm_{y,R} (r_z\nabla u)_{B(z,2r_z),p}^{q-p} \right) \bint_{B(z,2r_z)} |r_z\nabla u|^p \, \dx \bigg]
    \\ & \leq 
    c(n,p,q,c_\mathrm{fat}) \left[ 1 + a^\pm_{y,R} M_{2r_z}(|r_z\nabla u|^p )^{\frac{q-p}{p}}(z) \right] M_{2r_z}(|r_z\nabla u|^p )(z),
\end{align*}
i.e.~we have the claim of \eqref{thm:equivalence-4}.

Finally, we prove that \eqref{thm:equivalence-4} implies \eqref{thm:equivalence-1}.
At this stage, we assume that the constant on the right-hand side of \eqref{thm:equivalence-4} only depends on $n$, $p$, and $q$, but not on any additional number $c_\mathrm{fat}>0$.
We will then determine the constant $c_\mathrm{fat}=c_\mathrm{fat}(n,p,q)$ in \eqref{thm:equivalence-1}.
Observe that a change of variables as in \eqref{eq:equivalence-density-aux3} shows that the infimal $\varphi$-capacity density condition \eqref{eq:inf-capacity-density} is fulfilled
if there exists $c_\mathrm{fat}>0$ such that the inequality
\begin{equation}
    c_\mathrm{fat} \varphi^-_{z,2r}(t)
    \leq
    \bint_{B(z,2r)} \varphi(y,2r\nabla u) \,\dy
    \label{eq:aux-equivalence}
\end{equation}
holds for every point $z \in \R^n \setminus \Omega$, level $t>0$, radius $0<2r<r_o$ and $u\in \mathcal{M}_{z,2r}$, where

$$
    \mathcal{M}_{z,2r} :=
    \left\{
    u \in \Lip_0(\overline{B(z,2r)}) :
    u\geq t \chi_{(\R^n \setminus \Omega) \cap \overline{B(z,r)}}
    \right\}.
$$

Further, by Lemma \ref{lem:truncation}, we may assume that $0\leq u\leq t$.
At this stage, let us first consider the case where
\begin{equation}
    \bint_{B(z,r)} u \, \dy
    >
    \frac{t}{4^{n+1}}.
    \label{inequ-312}
\end{equation}
Note that for any $s\in [1,\infty)$, the classical Sobolev-Poincaré inequality (see for instance \cite[Theorem 3.17]{KLV21}) implies that
\begin{equation*}
    t <
    4^{n+1} \bint_{B(z,r)} u \, \dy
    \leq
    c(n) \bint_{B(z,2r)} u \, \dy
    \leq
    c(n,s) \bigg(  \bint_{B(z,2r)} |2r\nabla u|^s \, \dy  \bigg)^{\frac{1}{s}}.
\end{equation*}
Applying the preceding inequality with the parameter choices $s=p,q$, we obtain that
\begin{align*}
    \varphi^-_{z,2r}(t)
    &=
    \big( t^p + a_{z,2r}^- t^q \big)
    \leq
    c(n,p,q) \bint_{B(z,2r)} \varphi^-_{z,2r}(2r\nabla u) \, \dy
    \leq
    c(n,p,q) \bint_{B(z,2r)} \varphi(y,2r\nabla u) \, \dy,
\end{align*}
i.e.~that \eqref{eq:aux-equivalence} holds.
Thus, \eqref{thm:equivalence-1} immediately follows if \eqref{inequ-312} holds for every $u\in \mathcal{M}_{z,2r}$.

Let us now consider the case where \eqref{inequ-312} is not true for a function $u \in \mathcal{M}_{z,2r}$.
First, note that since $u$ is continuous, the sublevel set
$$
    F:=
    \left\{y\in B\left( z, \tfrac{r}{4}  \right) : 0\leq u< \tfrac{t}{2} \right\}
$$
is open.
Furthermore, we have that $F\subset\Omega$, since $u\geq t \chi_{(\R^n \setminus \Omega) \cap \overline{B(z,r)}}$, and $u\geq t/2$ on $B\left( z, \tfrac{r}{4}  \right)\setminus F$.
By these observations and the fact that \eqref{inequ-312} is false for $u$, for the measure of this set we obtain that
\begin{equation*}
    \left|B\left( z, \tfrac{r}{4}  \right)\setminus F\right|
    =
    \int_{B\left( z, \frac{r}{4}  \right)\setminus F} 1 \,\dy
    \leq
    \frac{2}{t} \int_{B\left( z, \frac{r}{4}  \right)\setminus F} u \,\dy
    \leq
    \frac{1}{2 \cdot 4^n} |B(z,r)|,
\end{equation*}
which in turn implies that
\begin{equation*}
    |F|
    =
    \left| B\left( z, \tfrac{r}{4}  \right) \right| - \left| B\left( z, \tfrac{r}{4}  \right) \setminus F \right|
    \geq
    \frac{1}{4^n} |B(z,r)| - \frac{1}{2 \cdot 4^n} |B(z,r)|
    =
    \frac{1}{2 \cdot 4^n} |B(z,r)|.
\end{equation*}
To proceed, we define the cut-off function $\psi \colon \R^n \to \R_{\geq 0}$ by
\begin{equation*}
    y\mapsto \max\big\{ 0, 1 - \tfrac{4}{r} \dist\big(y,B\big(z,\tfrac{3r}{4} \big)\big) \big\}
\end{equation*}
and set
$$
    v := \psi(t-u)
    \in \Lip(\overline{B(z,r)}).
$$
Note that $\spt(v) \Subset \overline{\Omega \cap B(z,r)}$, since $u\equiv t$ on $\overline{B(z,r)}\cap (\R^n \setminus \Omega)$.
Further, by construction of $\psi$, we have that $v=t-u$ in $B\big( z,\frac{3r}{4} \big)$, and hence that $|\nabla u|=|\nabla v|$ almost everywhere in $B\big( z,\frac{3r}{4} \big)$. 
By definition of the restricted centered maximal operator, we have that for every $x\in F$ there exists $0<r_x\leq 2 \dist(x,\partial\Omega)=: 2 d_x$ such that
\begin{equation}
\label{eq:aux_maxop}
    M_{2d_x}\big( |d_x\nabla u|^p \big)(x)
    =
    \sup_{0<r<2d_x} \bint_{B(x,r)} |d_x\nabla u|^p \, \dy
    \leq 
    2\bint_{B(x,r_x)} |d_x\nabla u|^p \, \dy.
\end{equation}
Next, since $z \in \R^n \setminus \Omega$ and $x \in F \subset \Omega \cap B\big(z, \frac{r}{4} \big)$, we have that $d_x \leq \frac{r}{4}$ and thus $r_x \leq \frac{r}{2}$.
Therefore, for any $y \in B(x,r_x)$ we find that
\begin{equation*}
    |z-y| \leq |z-x| + |x-y|
    \leq
    \frac{r}{4} + r_x
    \leq
    \frac{3r}{4}.
\end{equation*}
In particular, this leads to $B( x, r_x) \subset B\big( z, \frac{3r}{4} \big)$.
On the one hand, this yields $|\nabla u| = |\nabla v|$ almost everywhere on $B(x,r_x)$.
On the other hand, the universal bound for the radii $r_x$ allows us to apply the Vitali-type $5r$-covering lemma (see for instance \cite[Lemma 1.13]{KLV21}).
Therefore, there exists a collection of disjoint balls $B(x_i,r_i)$, $i \in \N$, with $x_i \in F$ and $r_i := r_{x_i}$, such that
\begin{equation*}
    F \subset \bigcup_{i=1}^{+\infty} B(x_i,5 r_i).
\end{equation*}       
Due to the lower bound on the measure of $F$, we deduce that
\begin{align*}
    |B(z,r)|
    \leq
    2 \cdot 4^n |F|
    \leq
    2 \cdot 20^n \sum_{i=1}^{+\infty} | B(x_i,r_i)|.
\end{align*}
At this stage, we apply the pointwise Hardy inequality \eqref{thm:equivalence-4}.
Since $x_i\in F$, observe that $v(x_i)=t-u(x_i)\geq t/2$ for any $i \in \N$ and that $r_i \leq 2d_{x_i} < 2r <r_o$.
Thus, applying \eqref{thm:equivalence-4} to $v(x_i)$, using \eqref{eq:aux_maxop}, Jensen's inequality and the fact that $B(x_i,r_i) \subset B(z,2r)$, and thus $a_{z,2r}^- \leq a_{x_i,r_i}^-$, for any $i \in \N$ we infer
\begin{align*}
    \varphi^-_{z,2r}( t )
    &\leq 
    2^q \varphi^-_{z,2r}\big( \tfrac{t}{2} \big)
    \leq 
    2^q \varphi^-_{z,2r}\big( v(x_i) \big) 
    \\ & \leq 
    c \left[ 1 + a^-_{z,2r} M_{2d_{x_i}}(|d_{x_i} \nabla u|^p )^{\frac{q-p}{p}}(x_i) \right] M_{2d_{x_i}}(|d_{x_i}\nabla u|^p )(x_i)
    \\ & \leq 
    c \left[ 1 + a^-_{z,2r} \bigg( \bint_{B(x_i,r_i)} |d_{x_i}\nabla u|^p \, \dy \bigg)^{\frac{q-p}{p}} \right]
    \bint_{B(x_i,r_i)} |d_{x_i}\nabla u|^p \, \dy
    \\ & \leq 
    c \, \bint_{B(x_i,r_i)} \left[|r\nabla u|^p + a^-_{x_i,r_i} |r\nabla u|^q \right] \, \dy
    \\ & \leq
    c \, \bint_{B(x_i,r_i)} \varphi(y,r\nabla u)\, \dy
\end{align*}
with a constant $c = c(n,p,q)$. From the previous chain of inequalities, we can immediately estimate the measures of the balls $B(x_i,r_i)$, $i \in \N$, by
\begin{align*}
    |B(x_i,r_i)| \varphi^-_{z,2r}( t )
    &\leq 
    c \int_{B(x_i,r_i)} \varphi(y,r\nabla u)\, \dy.
\end{align*}
Finally, using that the balls $B(x_i,r_i)$, $i \in \N$, are disjoint, the preceding inequality leads us to
\begin{align*}
    |B(z,r)| \varphi^-_{z,2r}( t )
    & \leq 
    c(n) \sum_{i=1}^{+\infty} | B(x_i,r_i)| \varphi^-_{z,2r}( t )
    \\ & \leq 
    c(n,p,q) \sum_{i=1}^{+\infty} \int_{B(x_i,r_i)} \varphi(y,r\nabla u)\, \dy
    \\ & \leq 
    c(n,p,q) \int_{B(z,2r)} \varphi(y,r\nabla u)\, \dy.
\end{align*}
Since this shows that \eqref{eq:aux-equivalence} holds true, we finally obtain that \eqref{thm:equivalence-1} is satisfied.
This concludes the proof of the theorem.
\end{proof}

\begin{remark}
\label{rmk:phases_equiv}
When considering the double-phase integrand \eqref{eq:integrand}, it is often necessary to distinguish between the $p$-phase, where $a$ is close to zero, and the $(p,q)$-phase, where $a$ is bounded away from zero.
More precisely, in a ball $B(z,r) \subset \R^n$ we consider the cases
\begin{equation}
	\left\{
	\begin{array}{ll}
		a_{z,r}^- \leq [a]_{0;\alpha} r^\alpha
		& \text{($p$-phase)}, \\[5pt]
		a_{z,r}^- > [a]_{0;\alpha} r^\alpha
		& \text{($(p,q)-$phase)}.
	\end{array}
	\right.
	\label{eq:phases}
\end{equation}
In the $p$-phase \eqref{eq:phases}$_1$, we have that
\begin{equation}
    a^+_{z,r}
    =
    a^-_{z,r} + |a^+_{z,r} - a^-_{z,r}|
    \leq
    [a]_{0;\alpha} r^\alpha + [a]_{0;\alpha} (2r)^\alpha
    \leq
    3[a]_{0;\alpha} r^\alpha,
    \label{eq:sup-osc-p-phase}
\end{equation}
whereas in the $(p,q)$-phase \eqref{eq:phases}$_2$, we find that
\begin{equation}
    a_{z,r}^+
    \leq
    a^-_{z,r} + [a]_{0;\alpha} (2r)^\alpha
    <
    3 a_{z,r}^-.
    \label{eq:sup-inf-pq-phase}
\end{equation}
    While we have stated Theorem~\ref{thm-equivalence} independent of the $p$- and $(p,q)$-phases defined in \eqref{eq:phases}, by distinguishing between these phases, we may derive the following inequalities from Theorem~\ref{thm-equivalence} \eqref{thm:equivalence-2}-\eqref{thm:equivalence-4}.
    Let us start with the boundary Poincaré inequality in \eqref{thm:equivalence-2} with parameters $(y,R) \equiv (z,r)$.
    On the one hand, in the $p$-phase \eqref{eq:phases}$_1$, by \eqref{eq:sup-osc-p-phase} and \eqref{eq:gap-exponents}, we find that
    \begin{align*}
        \bint_{B(z,r)} \varphi(x,|u|) \, \dx
        & \leq
        \bint_{B(z,r)} \varphi^+_{z,r}(|u|) \, \dx
        \leq
        c(n,p,q) \big( 1 + a^+_{z,r} (r\nabla u)_{B(z,r),p}^{q-p} \big)
        \bint_{B(z,r)} |r \nabla u|^p \,\dx
        \\& \leq
        c(n,p,q) \big( 1 + [a]_{0,\alpha} r^{\alpha-\frac{n(q-p)}{p}} \|r \nabla u\|_{L^p(B(z,r))}^{q-p} \big)
        \bint_{B(z,r)} |r \nabla u|^p \,\dx
        \\& \leq
        c(n,p,q) \big( 1 + [a]_{0,\alpha} \|r \nabla u\|_{L^p(B(z,r))}^{q-p} \big)
        \bint_{B(z,r)} |r \nabla u|^p \,\dx
    \end{align*}
    On the other hand, in the $(p,q)$-phase we use \eqref{eq:sup-inf-pq-phase} and Hölder's inequality to obtain that
    \begin{align*}
        \bint_{B(z,r)} & \varphi(x,|u|) \, \dx
        \leq
        \bint_{B(z,r)} \varphi^+_{z,r}(|u|) \, \dx
        \leq
        c(n,p,q) \big( 1 + a^+_{z,r} (r\nabla u)_{B(z,r),p}^{q-p} \big)
        \bint_{B(z,r)} |r \nabla u|^p \,\dx
        \\ & \leq
        c(n,p,q) \big( 1 + a^-_{z,r} (r\nabla u)_{B(z,r),p}^{q-p} \big)
        \bint_{B(z,r)} |r \nabla u|^p \,\dx
        \\& \leq
        c(n,p,q) \left[ \bint_{B(z,r)} |r \nabla u|^p \,\dx + a^-_{z,r} \bigg( \bint_{B(z,r)} |r \nabla u|^p \,\dx \bigg)^\frac{q}{p} \right]
        \\& \leq
        c(n,p,q) \bint_{B(z,r)} \left[|r \nabla u|^p + a^-_{z,r} |r \nabla u|^q \right] \,\dx
        \\& =
        c(n,p,q) \bint_{B(z,r)} \varphi^-_{z,r}(|r \nabla u|)\,\dx
        \leq
        c(n,p,q) \bint_{B(z,r)} \varphi(x,|r \nabla u|)\,\dx.
    \end{align*}
    Similarly, assuming that $r_z \leq 1$, in the $p$-phase we estimate the right-hand side of \eqref{thm:equivalence-3} by
    \begin{align*}
            \varphi\big(z,(u)_{B(z,r_z)} \big)
            & \leq
            \varphi^+_{z,r_z}\big((u)_{B(z,r_z)} \big)
             \leq
             c(n,p,q) \big( 1 +[a]_{0,\alpha}\|r_z\nabla u\|_{ L^p(B(z,2r_z))}^{q-p} \big)
            \bint_{B(z,2r_z)} |r_z \nabla u|^p \,\dx,
    \end{align*}
    whereas in the $(p,q)$-phase we obtain that
    \begin{align*}
            \varphi\big(z,(u)_{B(z,r_z)} \big)
            & \leq
            \varphi^+_{z,r_z}\big((u)_{B(z,r_z)} \big)
            \leq
            c(n,p,q) \bint_{B(z,2r_z)} \varphi^-_{z,r_z}(|r_z \nabla u|)\,\dx
            \\& \leq
            c(n,p,q) \bint_{B(z,2r_z)} \varphi(x,|r_z \nabla u|)\,\dx.
    \end{align*}
    Finally, choosing $(y,R) \equiv (z,2r_z)$, \eqref{thm:equivalence-4} is already the expected statement in the $p$-phase.
    However, in the $(p,q)$-phase, by Hölder's inequality and the monotonicity of $s\mapsto s^{p/q}$ we obtain that
    \begin{align*}
        \varphi( z, u(z))
        & \leq 
        \varphi^+_{z,2r_z}\big( u(z) \big)
        \leq
        c(n,p,q) \left[ 1 + a^+_{z,2r_z} M_{2r_z}(|r_z\nabla u|^p )^{\frac{q-p}{p}}(z) \right] M_{2r_z}(|r_z\nabla u|^p )(z)
        \\ & \leq
        c(n,p,q) \left[ 1 + a^-_{z,2r_z} M_{2r_z}(|r_z\nabla u|^p )^{\frac{q-p}{p}}(z) \right] M_{2r_z}(|r_z\nabla u|^p )(z)
        \\ & =
        c(n,p,q) \left[ M_{2r_z}(|r_z\nabla u|^p )(z) + a^-_{z,2r_z} \bigg( \sup\limits_{0<r\leq 2r_z} \bint_{B(z,r)} |r_z \nabla u|^p \,\dx \bigg)^{\frac{q}{p}} \right] 
        \\ & \leq
        c(n,p,q) \left[ M_{2r_z}(\varphi^-_{z,2r_z}(|r_z\nabla u|) )(z) + a^-_{z,2r_z} \bigg( \sup\limits_{0<r\leq 2r_z} \bigg( \bint_{B(z,r)} |r_z \nabla u|^q \,\dx \bigg)^{\frac{p}{q}} \bigg)^{\frac{q}{p}} \right] 
        \\ & \leq
        c(n,p,q) M_{2r_z}(\varphi^-_{z,2r_z}(|r_z\nabla u|) )(z)
        \leq
        c(n,p,q) M_{2r_z}(\varphi(\,\cdot\,,|r_z\nabla u|) )(z).
    \end{align*}
    In the $(p,q)$-phase, proceeding as in the proof of Theorem \ref{thm-equivalence} shows that the inequalities derived in this remark are even equivalent to the local infimal $\varphi$-fatness of $\R^n \setminus \Omega$.
\end{remark}


\section{Global higher integrability for functionals of double-phase type}
\label{sec:higher_int}

The focus of this section is on proving the global higher integrability result of Theorem \ref{thm:higher_int} and that is achieved by the following covering argument.
\\
We consider sufficiently small balls that cover $\Omega$ and distinguish between two types of balls, namely the ones that are completely contained in $\Omega$, and those that possess a nontrivial intersection with its complement. For these types of balls, we establish a reverse Hölder's inequality.
Then, we combine the respective inequalities on the individual balls to obtain a global reverse Hölder's inequality. 
To treat the balls $B$ that are not contained in $\Omega$, we consider a ball $B'$ that covers $B$ with center at the boundary $\partial\Omega$ and distinguish the cases where the center is sufficiently close to the zero set of $a$ in $\partial\Omega$ or not. 
In the first case, we obtain a reverse Hölder's inequality with the help of the locally uniform $p$-fatness, while in the other case we are still able to obtain a reverse Hölder's inequality, but making use of the locally uniform $q$-fatness.
For balls that are contained in $\Omega$, a reverse Hölder's inequality was already obtained;
for our purposes we combine arguments in the proof of \cite[Lemma 2.14]{Ok2017} with the claim of \cite[Theorem 2.13]{Ok2017} as stated in Lemma \ref{lem:poincare-zero-boundary-values}.

\begin{lemma}
\label{lem:interior-reverse-holder}
Let $\varphi$ be defined according to \eqref{eq:integrand} such that \eqref{as:a} and \eqref{eq:gap-exponents} hold with $1<p<q<+\infty$. Let $\Omega \subset \R^n$ be a bounded open set, $z \in \Omega$ and $0<r \leq 1$ be a radius such that $B(z,r)\subset \Omega$.
Further, assume that the functional $\mathcal{F}$ given by \eqref{Functional} satisfies \eqref{Functional-growth}, and that $u\in W^{1,p}(\Omega)$ is a local $C_Q$-minimizer of $\mathcal{F}$ in the sense of Definition \ref{def:quasi-minimizer}.
Then there exist an exponent $\theta_0 = \theta_0(n,p,q)\in (0,1)$ and a positive constant $c = c(n,p,q, \nu, L, C_Q)$ such that
\begin{equation*}
         \bint_{B(z,\frac{r}{2})} \varphi(x,|\nabla u| ) \, \dx
        \leq 
        c\big( 1 + [a]_{0,\alpha} \| \nabla u \|_{L^p(\Omega)}^{q-p} \big)
        \bigg( \bint_{B(z,r)} \varphi^\theta(x,|\nabla u| ) \, \dx \bigg)^{\frac{1}{\theta}}
    \end{equation*}
holds for every $\theta\in [\theta_0,1]$.
\end{lemma}

\begin{proof}
Let $\frac{1}{2}<t<s\leq 1$ and consider a cut-off function $\eta \in C^\infty_0(B(z,sr);[0,1])$ with $\eta \equiv 1$ in $B(z,tr)$ and $|\nabla \eta| \leq \frac{2}{(s-t)r}$.
Using $v := u-\eta(u-(u)_{B(z,r)})$ as a comparison map in \eqref{eq:quasi-minimizer}, recalling that \eqref{Functional-growth} holds, and using that $\nabla v \equiv 0$ in $B(z,tr)$ and $\nabla v = (1-\eta)\nabla u - \nabla \eta \big( u - (u)_{B(z,r)} \big)$ in $B(z,sr) \setminus B(z,tr)$, we obtain that
\begin{align*}
    & \, \nu \int_{B(z,tr)} \varphi(x,|\nabla u|) \,\dx
    \leq
    \nu \int_{B(z,sr)} \varphi(x,|\nabla u|) \,\dx\\
    \leq
    & \, \mathcal{F}[u;\spt(v-u)] 
    \leq
    C_Q \mathcal{F}[v;\spt(v-u)]
    \leq
    C_Q L \int_{B(z,sr)} \varphi(x,|\nabla v|) \,\dx \\
    \leq
    &\, C_Q 2^{q-1} L \bigg[ \int_{B(z,sr) \setminus B(z,tr)} \varphi(x,|\nabla u|) \,\dx
    + \int_{B(z,sr)} \varphi\left(x, \frac{2|u-(u)_{B(z,r)}|}{(s-t)r} \right) \,\dx \bigg] \\
    \leq
    &\, C_Q 2^{q-1} L \int_{B(z,sr) \setminus B(z,tr)} \varphi(x,|\nabla u|) \,\dx
    + \frac{C_Q 2^{2q-1} L}{(s-t)^q}\int_{B(z,r)} \varphi\left(x, \frac{|u-(u)_{B(z,r)}|}{r} \right) \,\dx.
\end{align*}
Adding $C_Q 2^{q-1} L \int_{B(z,tr)} \varphi(x,|\nabla u|) \,\dx$ on both sides and dividing the result by $\nu + C_Q2^{q-1}L$, we conclude that
\begin{align*}
    \int_{B(z,tr)} \varphi(x,|\nabla u|) \,\dx
    \leq
    \vartheta \int_{B(z,sr)} \varphi(x,|\nabla u|) \,\dx
    +\frac{c(q,\nu,L,C_Q)}{(s-t)^q}\int_{B(z,r)} \varphi\left(x, \frac{|u-(u)_{B(z,r)}|}{r} \right) \,\dx,
\end{align*}
where
\begin{equation*}
    \vartheta := \frac{C_Q 2^{q-1}L}{\nu + C_Q 2^{q-1}L}.
\end{equation*}
Since $\vartheta \in (0,1)$, by applying Lemma \ref{lem:iteration} to $\phi \colon \left[\frac{1}{2},1\right] \to [0,+\infty)$ given by $\phi(s) := \int_{B(z,sr)} \varphi(x,|\nabla u|) \,\dx$ we infer
$$
    \int_{B\left(z,\frac{r}{2}\right)} \varphi(x,|\nabla u|) \,\dx
    \leq
    c(q,\nu,L,C_Q)\int_{B(z,r)} \varphi\left(x, \frac{|u-(u)_{B(z,r)}|}{r} \right) \,\dx.
$$
Finally, we estimate the right-hand side of the preceding inequality by Lemma \ref{lem:poincare-zero-boundary-values}.
This yields the claim of the lemma.
\end{proof}

The first step towards a boundary reverse Hölder's inequality, namely a boundary Caccioppoli inequality, was already proven in \cite[Lemma 7.4]{Karp21} for minimizers.
While the lemma applies to a wide class of generalized $\Phi$-functions $\varphi$ (see Remark \ref{rmk:orlicz}) including the double-phase integrand, we just state it for the latter.
For the sake of completeness, we adapt the proof to the case of quasi-minimizers.

\begin{lemma}
\label{lem:boundary-caccioppoli}
Let $\varphi$ be defined according to \eqref{eq:integrand} such that \eqref{as:a} and \eqref{eq:gap-exponents} hold with $1<p<q<+\infty$.
Let $\Omega \subset \R^n$ be a bounded open set, $z\in \Omega$ and $r>0$ such that $B(z,r)\setminus \Omega \neq \emptyset$, and assume that the functional $\mathcal{F}$ given by \eqref{Functional} fulfills \eqref{Functional-growth}.
Further, for $f\in W^{1,\varphi(\cdot)}(\Omega)$ let $u\in f + W^{1,\varphi(\cdot)}_0(\Omega)$ be a local $C_Q$-minimizer in the sense of Definition \ref{def:quasi-minimizer}.
Then there exists a constant $c=c(n,q, \nu, L, C_Q)$ such that
    \begin{align*}
        \frac{1}{\left|B\left(z, \frac{r}{2} \right)\right|} \int_{B(z, \frac{r}{2}) \cap \Omega} &\varphi(x,|\nabla u| )\, \dx \\
        &\leq 
        \frac{c}{|B(z,r)|} \int_{B(z,r)\cap \Omega} \varphi\left(x,\frac{|u-f|}{r}\right) \, \dx
        +
        \frac{c}{|B(z,r)|} \int_{B(z,r)\cap \Omega} \varphi\left(x,|\nabla f|\right) \, \dx.
    \end{align*}
\end{lemma}

\begin{proof}
For $\frac{1}{2}<t<s\leq 1$, we consider a cut-off function $\eta \in C^\infty_0(B(z,sr);[0,1])$ with $\eta \equiv 1$ in $B(z,tr)$ and $|\nabla \eta| \leq \frac{2}{(s-t)r}$.
Choosing $v := u - \eta(u-f)$ as comparison map in \eqref{eq:quasi-minimizer}, and using that $\nabla v = \nabla f$ in $B(z,tr)$ and $\nabla v = (1-\eta)\nabla u + \eta \nabla f - \nabla \eta ( u - f)$ in $B(z,sr) \setminus B(z,tr)$, we obtain that
\begin{align*}
    & \, \nu \int_{B(z,tr) \cap \Omega} \varphi(x,|\nabla u|) \,\dx
    \leq
    \nu \int_{B(z,sr) \cap \Omega} \varphi(x,|\nabla u|) \,\dx\\
    \leq
    & \, \mathcal{F}[u;\spt(v-u)]
    \leq
    C_Q \mathcal{F}[v;\spt(v-u)]
    \leq
    C_Q L \int_{B(z,sr) \cap \Omega} \varphi(x,|\nabla v|) \,\dx \\
    \leq
    & \, C_Q 3^{q-1} L \bigg[
    \int_{(B(z,sr) \setminus B(z,tr)) \cap \Omega} \varphi(x,|\nabla u|) \,\dx
    + \frac{2^q}{(s-t)^q}\int_{B(z,r) \cap \Omega} \varphi\left(x,\frac{|u-f|}{r}\right) \,\dx
    \\&\phantom{=}
    + \int_{B(z,r) \cap \Omega} \varphi(x,|\nabla f|) \,\dx
    \bigg].
\end{align*}
By adding $C_Q 3^{q-1} L \int_{B(z,tr) \cap \Omega} \varphi(x,|\nabla u|) \,\dx$ on both sides and dividing the result by $\nu + C_Q 3^{q-1} L$, we deduce that
\begin{align*}
    \int_{B(z,tr) \cap \Omega} \varphi(x,|\nabla u|) \,\dx
    &\leq
    \vartheta \int_{B(z,sr) \cap \Omega} \varphi(x,|\nabla u|) \,\dx
    + \frac{c(q,\nu,L,C_Q)}{(s-t)^q}\int_{B(z,r) \cap \Omega} \varphi\left(x,\frac{|u-f|}{r}\right) \,\dx
    \\&\phantom{=}
    + c(q,\nu,L,C_Q) \int_{B(z,r) \cap \Omega} \varphi(x,|\nabla f|) \,\dx,
\end{align*}
where
$$
   \vartheta := \frac{C_Q 3^{q-1} L}{\nu + C_Q 3^{q-1} L} \in (0,1).
$$
At this stage, we conclude the claim of the lemma by applying Lemma \ref{lem:iteration} with $\phi \colon \left[\frac{1}{2},1\right] \to [0,+\infty)$, $\phi(s) := \int_{B(z,sr) \cap \Omega} \varphi(x,|\nabla u|) \,\dx$.
\end{proof}

With these prerequisites in mind, we can finally prove the global higher integrability, where the main remaining task is to combine the Maz'ya inequality of Lemma \ref{lem:protoMazya-type-ineq} with the self-improving property of the capacity, namely Theorem \ref{thm:pselfimproving2}, and the previous boundary Caccioppoli inequality.
\\[0.2cm]
\begin{proof}[Proof of Theorem \ref{thm:higher_int}]
    Let $z\in \Omega$. We distinguish two cases for the behavior of the coefficient $a$ on the boundary. If $a>0$ on the entirety of $\partial\Omega$, we set $P(r_a)=P(r_a/2)=\emptyset$ and $\hat{r}_0=r_0$. Otherwise, we define
    \begin{equation*}
        P(r_a):= (\{x\in \partial\Omega \ : \ a(x)=0\} + B(0,r_a))\setminus \Omega
    \end{equation*}
    and note, by assumption of the Theorem, that every point of $P(r_a)$ is locally uniformly $p$-fat, whereas $\R^n \setminus (\Omega\cup P(r_a))$ is locally uniformly $q$-fat for the parameters $c_\mathrm{fat}$ and $\hat{r}_o$ independently of the considered case. Let us furthermore define the possibly empty set $\mathcal{P}:=P(r_a/2)$.
    Since the set $\partial\Omega\setminus \mathcal{P}$ is compact, if it is non-empty, the minimum
    \begin{equation*}
        a_\mathcal{P} := \inf\limits_{\partial\Omega\setminus \mathcal{P}} a(x)
    \end{equation*}
    exists and we have that $a_\mathcal{P}>0$. Now, consider some arbitrary radius
    \begin{equation*}
        0<r\leq \frac{1}{8}\min\left\{ \hat{r}_o,\bigg(\frac{a_\mathcal{P}}{2(1+[a]_{0,\alpha})}\bigg)^{1/\alpha} \right\},
    \end{equation*}
    where we set $a_\mathcal{P}=\infty$ if $\mathcal{P}=\emptyset$. We start by establishing a reverse Hölder's inequality for the ball $B(z,r)$. Indeed, if $B(z,r)\subset \Omega$, thanks to Lemma \ref{lem:interior-reverse-holder} there is nothing left to prove.
    Otherwise, there exists $y\in B(z,r)\cap \partial \Omega$.
    Now, we consider the ball $B(y,2r)$ which contains $B(z,r)$ and extend the function $v:=u-f$ by zero to $\R^n$. Note that since $v \in W^{1,\varphi(\cdot)}_0(\Omega)$, we have that $v \in W^{1,p}_0(\Omega)$ and $v \in W^{1,q}(O)$ for every open $O$ such that $\overline{O}\cap \Omega$ is contained in the set $\{x\in\Omega\ : \ a(x)\neq 0\}$. Let us now establish the reverse Hölder's inequality at the boundary.
    To this end, we distinguish between two scenarios for the center $y$. \\
    \textbf{If there holds $\bm{y\in \mathcal{P}}$}, we have, thanks to the choice of $r$ and the definition of $\mathcal{P}$, that every point in the neighborhood $\overline{B(y,r_a/2)}\setminus \Omega$ of $y$ is locally uniformly $p$-fat in the sense of Definition \ref{def:inf-capacity-density}. Since $r\leq \hat{r}_0/8\leq r_a/8$, we have according to Theorem \ref{thm:pselfimproving2} that  
    \begin{equation*}
        \cp_{p-\varepsilon}((\R^n \setminus \Omega)\cap B(y,r);B(y,2r)) 
        \geq 
        c(n,p,c_\mathrm{fat}) \cp_{p-\varepsilon}(B(y,r);B(y,2r)),
    \end{equation*}
    where $0<\varepsilon(n,p,c_\mathrm{fat})\leq p-1$ is the constant of Theorem \ref{thm:pselfimproving2}. For said constant, we can assume without loss of generality that
    \begin{equation}\label{aux:thm1-comparisionvarepsilon}
        \varepsilon(n,p,c_\mathrm{fat})\leq \varepsilon(n,q,c_\mathrm{fat})
    \end{equation}
    due to \eqref{eq:gap-exponents}. We then obtain for the vanishing set $Z:=\{x\in B(y,r) : v(x)=0\}$ of $v$ the lower bound
    \begin{equation}
    \label{aux:thm1-1}
    \begin{aligned}
        \cp_{p-\varepsilon}(Z;B(y,2r)) 
        &\geq \cp_{p-\varepsilon}((\R^n \setminus \Omega)\cap B(y,r);B(y,2r)) \\
        &\geq c(n,p,c_\mathrm{fat}) \cp_{p-\varepsilon}(B(y,r);B(y,2r)) \\
        &\geq c(n,p,c_\mathrm{fat}) (2r)^{n-p+\varepsilon},
    \end{aligned}
    \end{equation}    
    where we have used the monotonicity of the variational capacity and the lower bound \eqref{eq:scaling-p-capacity}.\\
    In order to continue with the proof, we would like to apply Lemma \ref{lem:protoMazya-type-ineq} with a suitable parameter $m$, i.e.~a Maz'ya type inequality, to $v$.
    While the lemma is formulated for Lipschitz functions, an approximation argument shows that it continues to hold for functions $u \in W^{1,m}(\Omega)$.\\
    To this end, we have to distinguish between the cases $p\leq n$ and $p>n$. In the former case, we choose $p-\tfrac{p\varepsilon}{q}<m<p\leq n$ such that
    \begin{equation}\label{aux:thm1-choiceofm}
        m \geq p \, \theta_0,\qquad
        q\leq \frac{nm}{n-m}
        \qquad\text{and}\qquad
        q\leq p  + \frac{m\alpha}{n} ,
    \end{equation}
    while in the latter case we choose $n<m<p$ such that
    \begin{equation}\label{aux:thm1-choiceofm2}
        m > p - \frac{p\,\varepsilon}{q},\qquad
        m \geq p \, \theta_0
        \qquad\text{and}\qquad
        q\leq p  + \frac{m\alpha}{n},
    \end{equation}
    where $\theta_0 = \theta_0(n,p,q) \in (0,1)$ denotes the constant of Lemma \ref{lem:interior-reverse-holder}.
    Note that such a choice of $m=m(n,p,q,\alpha)$ is possible due to Remark \ref{rmk:choiceofm} and assumption \eqref{eq:gap-exponents-strict}.
    Moreover, observe that the latter upper bound on $q$ implies the inequality
    \begin{equation}
    \label{aux:ineq_1}
        \alpha-\frac{n(q-p)}{m}\geq 0.
    \end{equation}
    
    Now, notice that Hölder's inequality and \eqref{aux:thm1-1} imply
    \begin{equation*}
        \cp_m(Z;B(y,2r))\geq c(n,p,q,\alpha,c_\mathrm{fat}) (2r)^{n-m}
    \end{equation*}
    and therefore, after rescaling as in \eqref{eq:equivalence-density-aux}, that
    \begin{equation*}
        \cp_m\left(\left\{ x\in \overline{B(0,1/2)}: v(y+2rx)=0 \right\};B(0,1)\right)\geq c(n,p,q,\alpha,c_\mathrm{fat}).
    \end{equation*}
    If we define $\theta_1:=m/p \in [\theta_0,1)$, and apply Lemma \ref{lem:protoMazya-type-ineq} on $v$ with the parameter choices $(r,R,z,y)\equiv (2r,2r,y,y)$, thanks to the previous inequality we obtain that
    \begin{align*}
        \bint_{B(y,2r)\cap \Omega} & \varphi^+_{y,2r}\left(\frac{|v|}{r}\right) \,\dx
        \\ &\leq
        c \big[ 1 + a^+_{y,2r}  (\nabla v)_{B(y,2r),m}^{q-p} \big] 
        \bigg( \bint_{B(y,2r)} | \nabla v|^m \,\dx \bigg)^{\frac{p}{m}}
        \\ &=
        c \big[ 1 + a^+_{y,2r}  (\nabla v)_{B(y,2r),p\theta_1}^{q-p} \big] 
        \bigg( \bint_{B(y,2r)} | \nabla v|^{p\theta_1} \,\dx \bigg)^{\frac{1}{\theta_1}}
    \end{align*}
    for a constant $c=c(n,p,q,\alpha,c_\mathrm{fat})$. Let us now estimate the right-hand side further using the inequality 
    \begin{equation*}
        a^+_{y,2r} \leq a^-_{y,2r} + [a]_{0,\alpha} (2r)^\alpha.
    \end{equation*}
    We observe in a first step that Jensen's inequality and the definition of $a^-_{y,2r}$ imply that
    \begin{align*}
        a^-_{y,2r} &(\nabla v)_{B(y,2r),p\theta_1}^{q-p} \bigg( \bint_{B(y,2r)} | \nabla v|^{p\theta_1} \,\dx \bigg)^{\frac{1}{\theta_1}} \\
        = \, & a^-_{y,2r} \bigg( \bint_{B(y,2r)} | \nabla v|^{p\theta_1} \,\dx \bigg)^{\frac{q}{p \theta_1}}
        \leq
        a^-_{y,2r} \bigg( \bint_{B(y,2r)} | \nabla v|^{q\theta_1} \,\dx \bigg)^{\frac{1}{\theta_1}} \\
        = \,
        & \bigg( \bint_{B(y,2r)} \big(a^-_{y,2r}\big)^{\theta_1}  | \nabla v|^{q\theta_1} \,\dx \bigg)^{\frac{1}{\theta_1}}
        \leq
        \bigg( \bint_{B(y,2r)} \varphi^{\theta_1}(x, |\nabla v|) \,\dx \bigg)^{\frac{1}{\theta_1}} \\
        \leq \,
        & c \bigg( \frac{1}{|B(y, 2r)|} \int_{B(y,2r)\cap\Omega} \left[\varphi^{\theta_1}(x, |\nabla u|) + \varphi^{\theta_1}(x, |\nabla f|) \right] \,\dx \bigg)^{\frac{1}{\theta_1}} \\
        \leq \,
        & c \bigg( \frac{1}{|B(y,2r)|} \int_{B(y,2r)\cap\Omega} \varphi^{\theta_1}(x, |\nabla u|) \,\dx \bigg)^{\frac{1}{\theta_1}} 
        + 
        c \, \frac{1}{|B(y,2r)|} \int_{B(y,2r)\cap\Omega} \varphi(x, |\nabla f|) \,\dx.
    \end{align*}
    with a constant $c=c(n,p,q,\alpha,c_\mathrm{fat})$. 
    Next, due to 
    \eqref{aux:ineq_1} and the fact that $2r\leq 1$ we observe that
    \begin{align*}
        [a]_{0,\alpha} (&2r)^\alpha  (\nabla v)_{B(y,2r),p\theta_1}^{q-p} \bigg( \bint_{B(y,2r)} | \nabla v|^{p\theta_1} \,\dx \bigg)^{\frac{1}{\theta_1}} \\
        &=
        [a]_{0,\alpha} (2r)^{\alpha-\frac{n(q-p)}{m}} \| \nabla v\|_{L^{p\theta_1}(B(y,2r))}^{q-p} \bigg( \bint_{B(y,2r)} | \nabla v|^{p\theta_1} \,\dx \bigg)^{\frac{1}{\theta_1}} \\
        &\leq
        c \, [a]_{0,\alpha} 
        \| \nabla (u-f)\|_{L^{p}(\Omega)}^{\theta_1(q-p)}
        \\ & \qquad\cdot
        \Bigg[\bigg( \frac{1}{|B(y,2r)|} \int_{B(y,2r)\cap\Omega} \varphi^{\theta_1}(x, |\nabla u|) \,\dx \bigg)^{\frac{1}{\theta_1}} 
        + 
        \frac{1}{|B(y,2r)|} \int_{B(y,2r)\cap\Omega} \varphi(x, |\nabla f|) \,\dx  \Bigg],
    \end{align*}
    with a constant $c=c(n,p,q,\alpha,c_\mathrm{fat})$. Combining the preceding estimates and using the fact that $\varphi(x,\cdot)\leq \varphi^+_{y,2r}(\cdot)$ on $B(y,2r)$, we get that
    \begin{align}\label{aux:thm1-2}
        \frac{1}{|B(y,2r)|} & \int_{B(y,2r)\cap \Omega} \varphi\left(x,\frac{|u-f|}{r}\right) \,\dx
        \nonumber \\ &\leq
        c \Big[ 1
        + [a]_{0,\alpha}
        \| \nabla (u-f) \|_{L^{p}(\Omega)}^{\theta_1(q-p)}
        \Big] 
        \Bigg[\bigg( \frac{1}{|B(y,2r)|} \int_{B(y,2r)\cap\Omega} \varphi^{\theta_1}(x, |\nabla u|) \,\dx \bigg)^{\frac{1}{\theta_1}} 
        \nonumber\\ & \qquad+ 
        \frac{1}{|B(y,2r)|} \int_{B(y,2r)\cap\Omega} \varphi(x, |\nabla f|) \,\dx  \Bigg]
    \end{align}
    with a constant $c=c(n,p,q,\alpha,c_\mathrm{fat})$.
    \\
    \textbf{Now, assume on the contrary that there holds $\bm{y\in \partial\Omega\setminus \mathcal{P}}$}, which can only be the case if the set is non-empty. To treat this case, we observe that $\R^n\setminus \Omega$ is locally uniformly $q$-fat and, accordingly, Theorem \ref{thm:pselfimproving2} in conjunction with \eqref{aux:thm1-comparisionvarepsilon} yields that $\R^n\setminus \Omega$ is locally uniformly $(q-\varepsilon)$-fat. Furthermore, we have in similarity to \eqref{aux:thm1-1} that the estimate 
    \begin{equation}
    \label{aux:thm1-3}
    \begin{aligned}
        \cp_{q-\varepsilon}(Z;B(y,2r)) 
        &\geq c(n,p,c_\mathrm{fat}) (2r)^{n-q+\varepsilon}
    \end{aligned}
    \end{equation}
    holds. But this allows us to apply the classical Maz'ya inequality of Lemma \ref{lem:classical-Mazya} on $v$ with parameters $(p,s)\equiv (q\theta_1,q)$, where $m$ is again chosen according to \eqref{aux:thm1-choiceofm} if $p \leq n$ and according to \eqref{aux:thm1-choiceofm2} if $p>n$, respectively, and $\theta_1 := m/p$. To ensure its applicability, we need to consider the case $q\theta_1 <n$ and verify the additional condition that $q\leq n\theta_1 q / ( n-\theta_1 q)$. By definition of $\theta_1$, we have $m=\theta_1 p$ which implies $m<n$. But this contradicts the choice \eqref{aux:thm1-choiceofm2} for $m$ in the case $p>n$, i.e., that $n<m<p$. Thus, there must hold $p\leq n$. Now, we can use the second inequality of \eqref{aux:thm1-choiceofm} to derive the estimate
    \begin{equation*}
        n-q\theta_1 \leq n - p\theta_1 \leq \frac{n p \theta_1}{q} \leq n\theta_1
    \end{equation*}
    which yields the desired upper bound. Since $q \theta_1 > q - \varepsilon$, we further observe that Hölder's inequality and the inequality \eqref{aux:thm1-3} imply that \eqref{aux:thm1-3} also holds if $q-\varepsilon$ is replaced with $\theta_1 q$. If we use this to estimate the right-hand side of Lemma \ref{lem:classical-Mazya}, we observe that
    \begin{align*}
        \bint_{B(y,2r)} \bigg(\frac{|v|}{r}\bigg)^q \, \dx 
        \leq 
        c(n,p,q,\alpha) \bigg( \bint_{B(y,2r)} |\nabla v|^{q\theta_1} \, \dx\bigg)^{\frac{1}{\theta_1}} .
    \end{align*}
    By the assumptions on $r$ we have that $[a]_{0,\alpha} (2r)^\alpha < a_\mathcal{P}$. This allows us to use the preceding estimate and, in case $q>p$, Young's inequality with exponents $(\tfrac{q}{p},\tfrac{q}{q-p})$ to obtain 
    \begin{align}\label{aux:thm1-4}
        \frac{1}{|B(y,2r)|} & \int_{B(y,2r)\cap \Omega}  \varphi\left(x,\frac{|u-f|}{r}\right) \,\dx
        \nonumber \\ & \leq
        (a_\mathcal{P}+a^-_{y,2r}+[a]_{0,\alpha} (2r)^\alpha ) \bint_{B(y,2r)} \bigg(\frac{|v|}{r}\bigg)^q \, \dx +  a_\mathcal{P}^{-\frac{p}{q-p}}
        \nonumber \\ & \leq
        3c a^-_{y,2r} \bigg( \bint_{B(y,2r)} |\nabla v|^{q\theta_1} \, \dx\bigg)^{\frac{1}{\theta_1}}
        +
        a_\mathcal{P}^{-\frac{p}{q-p}}
        \nonumber \\ & \leq
        3 c \bigg( \bint_{B(y,2r)} \varphi^{\theta_1}(x,|\nabla v|) \, \dx\bigg)^{\frac{1}{\theta_1}} 
        +
        a_\mathcal{P}^{-\frac{p}{q-p}},
    \end{align}
    with $c=c(n,p,q,\alpha)$. We note that the same estimate also holds for $q=p$.
    \\[1em]
    {\bf Now, we proceed with both cases simultaneously.} To this end, we combine \eqref{aux:thm1-2} and \eqref{aux:thm1-4} and we obtain the estimate
    \begin{align*}
        \frac{1}{|B(y,2r)|} & \int_{B(y,2r)\cap \Omega} \varphi\left(x,\frac{|u-f|}{r}\right) \,\dx
        \nonumber \\ &\leq
        c \Big[ 1
        + [a]_{0,\alpha}
        \| \nabla (u-f) \|_{L^{p}(\Omega)}^{\theta_1(q-p)}
        \Big] 
        \Bigg[\bigg( \frac{1}{|B(y,2r)|} \int_{B(y,2r)\cap\Omega} \varphi^{\theta_1}(x, |\nabla u|) \,\dx \bigg)^{\frac{1}{\theta_1}} 
        \nonumber\\ & \qquad+ 
        \frac{1}{|B(y,2r)|} \int_{B(y,2r)\cap\Omega} \varphi(x, |\nabla f|) \,\dx \Bigg] 
        +
        a_\mathcal{P}^{-\frac{p}{q-p}}.
    \end{align*}
    If we further apply the Caccioppoli inequality of Lemma \ref{lem:boundary-caccioppoli} together the previous estimate, we infer the estimate
    \begin{align*}
        \frac{1}{|B\left(z, \frac{r}{2}\right)|} & \int_{B\left(z, \frac{r}{2}\right)\cap \Omega} \varphi(x,|\nabla u| )\, \dx
        \\ & \leq 
        \frac{c}{|B(z,r)|} \int_{B(z,r)\cap \Omega} \varphi\left(x,\frac{|u-f|}{r}\right) \, \dx
        +
        \frac{c}{|B(z,r)|} \int_{B(z,r)\cap \Omega} \varphi\left(x,|\nabla f|\right) \, \dx
        \\ & \leq 
        \frac{c}{|B(y,2r)|} \int_{B(y,2r)\cap \Omega} \varphi\left(x,\frac{|u-f|}{r}\right) \, \dx
        +
        \frac{c}{|B(z,3r)|} \int_{B(z,3r)\cap \Omega} \varphi\left(x,|\nabla f|\right) \, \dx
        \\ & \leq 
        c\Big[ 1+ [a]_{0,\alpha}\| \nabla (u-f) \|_{L^{p}(\Omega)}^{\theta_1(q-p)}\Big] \Bigg[ \bigg( \frac{1}{|B(z,3r)|} \int_{B(z,3r)\cap\Omega} \varphi^{\theta_1}(x, |\nabla u|) \,\dx \bigg)^{\frac{1}{\theta_1}}
        \\ & \quad +
        \frac{c}{|B(z,3r)|} \int_{B(z,3r)\cap \Omega} \varphi\left(x,|\nabla f|\right) \, \dx
        +
        a_\mathcal{P}^{-\frac{p}{q-p}}\bigg]
        \\ & \leq 
        c\Big[ 1+ [a]_{0,\alpha} \big( \| \nabla u \|_{L^{p}(\Omega)} + \| \nabla f \|_{L^{p}(\Omega)} \big)^{\theta_1(q-p)}\Big] \Bigg[ \bigg( \frac{1}{|B(z,3r)|} \int_{B(z,3r)\cap\Omega} \varphi^{\theta_1}(x, |\nabla u|) \,\dx \bigg)^{\frac{1}{\theta_1}}
        \\ & \quad +
        \frac{c}{|B(z,3r)|} \int_{B(z,3r)\cap \Omega} \varphi\left(x,|\nabla f|\right) \, \dx
        +
        a_\mathcal{P}^{-\frac{p}{q-p}} \bigg],
    \end{align*}
    with a constant $c=c(n,p,q,\alpha,[a]_{0,\alpha},\nu,L,c_\mathrm{fat},C_Q)$.
    Note that, due to Lemma \ref{lem:interior-reverse-holder}, the previous inequality also holds if $B(r,z)\subset\Omega$, and therefore, for every $z\in \Omega$.
    Let us now define the functions $g,h\colon \R^n \to \R_{\geq 0}$ as
    \begin{equation*}
        g(x) =\begin{cases}
            \varphi(x,|\nabla u| ) & \textnormal{for } x\in \Omega,\\
            0 & \text{else},
        \end{cases}
    \end{equation*}
    and
    \begin{equation*}
        h(x) = a_\mathcal{P}^{-\frac{n}{p\alpha}} + \begin{cases}
            \varphi(x,|\nabla f| )  & \text{for } x\in \Omega,\\
            0 & \text{else}.
        \end{cases}
    \end{equation*}
    Then the preceding shows
    \begin{align*}
        \bint_{B\left(z,\frac{r}{2}\right)} & g \, \dx
         \\ & \leq
         c\Big[ 1+ [a]_{0,\alpha} \big( \| \nabla u \|_{L^{p}(\Omega)} + \| \nabla f \|_{L^{p}(\Omega)} \big)^{\theta_1(q-p)}
         \Big] \Bigg[\bigg( \bint_{B(z,3r)} g^{\theta_1} \,\dx \bigg)^{\frac{1}{\theta_1}}
        +
         \bint_{B(z,3r)} h \, \dx \Bigg].
    \end{align*}
    Applying Gehring's Lemma, i.e.~Lemma \ref{Gehring's Lemma}, yields the existence of a constant $c$ and exponent $1<\sigma\leq \sigma_0$, both depending on $n,p,q,\alpha,[a]_{0,\alpha},\nu,L,c_\mathrm{fat},C_Q,\sigma_0$ and$\| \nabla u \|_{L^{p}(\Omega)}+ \| \nabla f \|_{L^{p}(\Omega)}$, such that
    \begin{align}
        \bigg(\bint_{B\left(z,\frac{r}{2}\right)} g^\sigma \, \dx \bigg)^{\frac{1}{\sigma}}
        \leq 
        c \Bigg[ \bint_{B(z,4r)} g \,\dx 
        +
        \bigg( \bint_{B(z,4r)} h^\sigma \, \dx \bigg)^{\frac{1}{\sigma}} \Bigg].
        \label{eq:higher-int-aux}
    \end{align}
    At this stage, we prove global higher integrability by the following covering argument.
    Let
    \begin{equation*}
        \varrho = \frac{1}{8}\min\left\{\hat{r}_o,\bigg(\frac{a_\mathcal{P}}{2(1+[a]_{0,\alpha})}\bigg)^{1/\alpha} \right\}.
    \end{equation*}
    Consider the collection of balls
    \begin{equation*}
        \mathcal{B} = \left\{ B\left(z,\frac{\varrho}{10} \right) \ : \ z\in \Omega \right\}
    \end{equation*}
    covering $\overline{\Omega}$.
    By Vitali's covering theorem (see for instance \cite[Lemma 1.13]{KLV21} for a version using open balls) and the fact that $\overline{\Omega}$ is compact, there exist $k\in \N$ and points $z_1,\ldots,z_k\in \Omega$ such that the balls $B\left( z_i, \frac{\varrho}{10} \right)$ for $i=1,\ldots,k$ are disjoint and
    $$
        \overline{\Omega}
        \subset
        \bigcup_{i=1}^k B\left(z_i, \frac{\varrho}{2} \right).
    $$
    We have chosen this approach since it allows us to determine $k$ in terms of $n$, $\Omega$ and $\varrho$.
    Indeed, since $B\left( z_i, \frac{\varrho}{10} \right) \subset B\left( z_o, \frac{\varrho}{10} + \diam(\Omega) \right)$ for any $z_o \in \Omega$ and $i=1,\ldots,k$ and since the balls $B\left( z_i, \frac{\varrho}{10} \right)$ are pairwise disjoint, we find that
    \begin{equation*}
        k |B(0,1)| \left(\frac{\varrho}{10}\right)^n
        =
        \sum_{i=1}^k \left| B\left( z_i, \frac{\varrho}{10} \right) \right|
        \leq 
        \left| B\left( z_o, \frac{\varrho}{10} + \diam(\Omega) \right) \right|
        = 
        |B(0,1)|\left( \frac{\varrho}{10} + \diam(\Omega) \right)^n.
    \end{equation*}
    Therefore, we have that
    \begin{equation*}
        k \leq \left( 1 + \frac{10 \diam(\Omega)}{\varrho}\right)^n .
    \end{equation*}
    Now, we apply \eqref{eq:higher-int-aux} on the individual balls $B\left( z_i, \frac{\varrho}{2} \right)$.
    Taking also the preceding estimate for $k$ into account and recalling the definitions of $g$ and $h$, we obtain that
    \begin{align*}
        \bigg(\int_\Omega &\varphi^\sigma(x,|\nabla u| ) \, \dx \bigg)^{\frac{1}{\sigma}}
        \leq
        \sum_{i=1}^k \bigg( \int_{B\left( z_i, \frac{\varrho}{2} \right)} g^\sigma \, \dx \bigg)^{\frac{1}{\sigma}}
        \\& \leq
        c \sum_{i=1}^k \Bigg[ \varrho^{\frac{(1-\sigma)n}{\sigma}} \int_{B(z,3\varrho)} g \,\dx 
        +
        \bigg( \int_{B(z,3\varrho)} h^\sigma \, \dx \bigg)^{\frac{1}{\sigma}} \Bigg]
        \\ & \leq 
        c k \Bigg[ \varrho^{\frac{(1-\sigma)n}{\sigma}} \int_{\Omega}  \varphi(x,|\nabla u|) \,\dx 
        +
        \bigg( \int_{\Omega} \varphi^\sigma(x,|\nabla f|) \, \dx \bigg)^{\frac{1}{\sigma}}
        +
        \max\Big\{ 1 , a_\mathcal{P}^{-\frac{p}{q-p}} \Big\} \Bigg]
        \\ & \leq 
        c \left( 1 + \frac{\diam(\Omega)}{\varrho}\right)^n
        \nonumber\\ & \qquad\cdot 
        \Bigg[ \varrho^{\frac{(1-\sigma)n}{\sigma}} \int_{\Omega}  \varphi(x,|\nabla u|) \,\dx 
        +
        \bigg( \int_{\Omega} \varphi^\sigma(x,|\nabla f|) \, \dx \bigg)^{\frac{1}{\sigma}} 
        +
        a_\mathcal{P}^{-\frac{p}{q-p}} \Bigg].
    \end{align*}
    This concludes the proof of the theorem.
\end{proof}

\begin{remark}
\label{rem:optimality}
    If $a>0$ on the entirety of $\partial\Omega$, the assumption on $\R^n \setminus \Omega$ in Theorem \ref{thm:higher_int} reduces to $q$-fatness, and if $a \equiv 0$ on $\partial\Omega$, we assume that $\R^n \setminus \Omega$ is locally $p$-fat.
    According to \cite{BP10}, in both of these cases, the respective fatness condition is optimal.
    Moreover, the following example shows that we have to assume local $p$-fatness of $\R^n \setminus \Omega$ at least at the boundary points where $a$ vanishes.
    Indeed, let $0<p<n$ and let $K \subset B(0,1)$ be a compact set such that $0 \in K$,
    \begin{equation*}
        \cp_p(K,B(0,2))>0 
        \qquad\text{and}\qquad
        \dim_\mathcal{H} (K) = n-p,
    \end{equation*}
    where $\dim_\mathcal{H}$ denotes the Hausdorff dimension.
    Further, set $\Omega=B(0,2)\setminus K$, consider a function $g\in C^\infty_0(B(0,2))$ with $g\equiv 1$ on $K$ and let $u$ denote the minimizer of the $p$-energy functional
    \begin{equation*}
        v \mapsto \int_\Omega |Dv|^p \, \dx
    \end{equation*}
    in the class $g+W^{1,p}_0(\Omega)$.
    Since $u$ is a weak solution to the $p$-Laplacian, $\partial B(0,2)$ is smooth and $g$ vanishes in the vicinity of $\partial B(0,2)$, we have that $u\in C^{1,\beta}_\mathrm{loc}(\Omega\cup \partial B(0,2))$ for some $\beta>0$, see e.g.~\cite{H92,L83,U77,U68}, and therefore
    \begin{equation*}
        \sup\limits_{x\in \overline{B(0,2)}\setminus B(0,1)} |Du|(x) < \infty.
    \end{equation*}
    Furthermore, for any ball $B(x,4r)\subset \Omega$ we have the quantitative estimate
    \begin{equation*}
        \sup\limits_{z\in B(x,r)} |Du|^p(z) \leq C(n,p) \frac{1}{r^n} \int_{B(x,2r)} |Du|^p \, \dz.
    \end{equation*}
    Now let $r=\dist(x,\partial \Omega)/4$ in the above formula and note that  $0\in K$ implies that $\dist(x,\partial \Omega) = \dist(x,K)$ for every $x\in B(0,1)\setminus K$. Then, we obtain that
     \begin{align*}
         |Du|^p(x)& \leq \sup\limits_{z\in B(x,r)} |Du|^p(z) 
         \leq 
         \frac{C(n,p)}{\dist(x,\partial\Omega)^n} \int_{B(x,2r)} |Du|^p \, \dz
         \\ & \leq
         \frac{ C(n,p)}{\dist(x,\partial\Omega)^n} \int_{\Omega} |Du|^p \, \dz
         =
         \frac{ C(n,p) \|Du\|^p_{p,\Omega}}{\dist(x,K)^n}.
    \end{align*}
    However, by \cite[Remark 3.3]{KK94}, see also \cite[Example 12.9]{KLV21} for an alternative proof, $u$ is not globally integrable to any higher exponent than $p$, i.e.~there is no $\delta>0$ with $Du\in L^{p+\delta}(\Omega)$.
    Now, we consider the double-phase integral with Lipschitz continuous coefficient $a(x):=\dist(x,K)$ and any $p<q< p(1+1/n)$.
    By the preceding estimate, we find that
    \begin{equation*}
         \kappa:= \sup_{x\in \Omega} a(x) |Du|^{q-p}(x) < \infty.
    \end{equation*}
    Further, since $u$ is a minimizer of the $p$-energy functional, we obtain that
    \begin{align*}
        \mathcal{F}[u;\spt(u-v)] &= \int_{\spt(u-v)} (|Du|^p + a(x) |Du|^q) \, \dx
        \leq
        (1+ \kappa) \int_{\spt(u-v)} |Du|^p \, \dx
        \\ & \leq
        (1+ \kappa) \int_{\spt(u-v)} |Dv|^p \, \dx
        \leq
        (1+\kappa) \mathcal{F}[v;\spt(u-v)],
    \end{align*}
    for any $v\in W^{1,p}(\Omega)$ with $\spt(u-v)\subset \Omega$.
    This shows that $u$ is a quasi-minimizer of $\mathcal{F}$ with $C_Q= 1 +\kappa$ and thus every assumption of Theorem \ref{thm:higher_int} but the local $p$-fatness condition in a neighborhood of $\{x \in \partial\Omega : a(x)=0\}$ is satisfied.
    Thus, in general the conclusion of Theorem \ref{thm:higher_int} does not hold without this fatness condition, because we would otherwise get a contradiction to the fact that $Du\notin L^{p+\delta}(\Omega)$ for any $\delta>0$.
\end{remark}


\bigskip
\textbf{Acknowledgments.}
This research was funded in part by the Austrian Science Fund (FWF) projects \href{https://doi.org/10.55776/F65    }{10.55776/F65} and \href{https://doi.org/10.55776/Y1292    }{10.55776/Y1292}.
This research was funded in whole or in part by the Austrian Science Fund (FWF) \href{https://doi.org/10.55776/J4853    }{10.55776/J4853}  and \href{https://doi.org/10.55776/P36295  }{10.55776/P36295}.
The author S.R.~would also like to thank the Department of Mathematics at Paris Lodron Universität  Salzburg for the hospitality during his visits.
All authors would also like to thank Univ.-Prof.~Dr.~Verena B\"ogelein for the helpful discussions.
For the purpose of open access, the authors have applied a CC BY public copyright license to any Author Accepted Manuscript (AAM) version arising from this submission.

\bigskip
\textbf{Data availability.} Data availability/data sharing is not applicable to this article as no datasets were generated or analyzed during the current study.

\bigskip
\textbf{Compliance with Ethical Standards.} The authors declare that they have no conflicts of interest. This article does not contain any studies with human participants or animals performed by the authors.


\end{document}